\documentclass[sn-mathphys-num]{sn-jnl}

\usepackage{graphicx}%
\usepackage{multirow}%
\usepackage{amsmath,amssymb,amsfonts}%
\usepackage{amsthm}%
\usepackage{mathrsfs}%
\usepackage[title]{appendix}%
\usepackage{xcolor}%
\usepackage{textcomp}%
\usepackage{manyfoot}%
\usepackage{booktabs}%
\usepackage{algorithm}%
\usepackage{algorithmicx}%
\usepackage{algpseudocode}%
\usepackage{listings}%
\usepackage{listings}%
\usepackage{supertabular}%
\usepackage{caption}%
\usepackage{float}%
\usepackage{indentfirst}%
\usepackage{extarrows}%
\usepackage{subfigure}%
\usepackage{longtable}%
\usepackage{bm}%
\usepackage{hyperref}%
\usepackage{array}%
\usepackage{multirow}%
\newtheorem{assumption}{Assumption}%
\newtheorem{lemma}{Lemma}%

\theoremstyle{thmstyleone}%
\newtheorem{theorem}{Theorem}
\newtheorem{proposition}[theorem]{Proposition}%

\theoremstyle{thmstyletwo}%

\theoremstyle{thmstylethree}%
\newtheorem{definition}{Definition}%

\begin{document}

\title[Article Title]{On Wasserstein Distributionally Robust Mean Semi-Absolute Deviation Portfolio Model: Robust Selection and Efficient Computation}


\author[1]{\fnm{Weimi} \sur{Zhou}}\email{wmzhou1997@163.com}

\author*[2]{\fnm{Yong-Jin} \sur{Liu}}\email{yjliu@fzu.edu.cn}
\equalcont{These authors contributed equally to this work.}

\affil[1]{\orgdiv{School of Mathematics and Statistics}, \orgname{Fuzhou University}, \orgaddress{\street{No. 2 Wulongjiang North Avenue}, \city{Fuzhou}, \postcode{350108}, \state{Fujian}, \country{China}}}

\affil*[2]{\orgdiv{Center for Applied Mathematics of Fujian Province, School of Mathematics and Statistics}, \orgname{Fuzhou University}, \orgaddress{\street{No. 2 Wulongjiang North Avenue}, \city{Fuzhou}, \postcode{350108}, \state{Fujian}, \country{China}}}


\abstract{This paper focuses on the Wasserstein distributionally robust mean-lower semi-absolute deviation (DR-MLSAD) model, where the ambiguity set is a Wasserstein ball centered on the empirical distribution of the training sample.  This model can be equivalently transformed into a convex  problem. We develop a robust Wasserstein profile  inference (RWPI) approach to determine the size of the Wasserstein radius for DR-MLSAD model. We also design an efficient proximal point dual semismooth Newton ({\sc PpdSsn}) algorithm  for the reformulated equivalent model. In numerical experiments, we compare the DR-MLSAD  model with the radius selected by the RWPI approach to the DR-MLSAD model with the radius selected by cross-validation, the sample average approximation (SAA) of the MLSAD model, and the $1/N$ strategy on  the real market datasets. Numerical results show that our model has better out-of-sample performance in most cases. Furthermore, we compare {\sc PpdSsn} algorithm with  first-order algorithms and Gurobi solver on  random data. Numerical results verify the effectiveness of {\sc PpdSsn} in solving large-scale DR-MLSAD problems.}




\maketitle

\section{Introduction}\label{Intro}

In modern portfolio, mean-variance model proposed by Markowitz \cite{Mark.1952} is a popular portfolio model to reflect the optimal trade-off between returns and risks. Due to the difficulty in calculating large-scale covariance matrix, more attention has been paid to the application of linear programming in portfolio. Konno and Yamazak \cite{Kon.Yama.1991}  proposed a simpler linear portfolio selection model called mean-absolute deviation (MAD) model, which reflects the deviation of portfolio returns below expected value and above expected value. Compared with mean-variance model, MAD model has less computational difficulty in dealing with large-scale portfolio selection. However, in the real financial investment market, rational investors are more concerned about the downside loss risk \cite{JR.ZF.2006}. Speranza \cite{Spera.1993} first proposed the mean-lower semi-absolute deviation (MLSAD) model, which minimizes the downside risk of an investment at an acceptable level of expected return. Assume that there are $m$ assets, random vector $\xi\in\Xi$ represents the asset returns in a certain period, then the precise formulation of the MLSAD model is 
\begin{equation}\label{MLSAD}
	\begin{aligned}
		&\min_{x\in\mathcal{X}}\ \mathbb{E}_{\mathbb{P}^*}|\min(0,\xi^{\top}x-\mathbb{E}_{\mathbb{P}^*}(\xi^{\top}x))|\\
		&\ {\rm{s.t.}}\quad \mathbb{E}_{\mathbb{P}^*}(\xi^{\top}x)=\bar{\rho}, 
	\end{aligned}
\end{equation}
where $\mathcal{X}=\{x|e^{\top}x=1,x\geq0\}$, ${\mathbb{P}^*}$ denotes the true probability distribution of the random vector $\xi$, $\mathbb{E}_{\mathbb{P}^*}$ denotes the expectation under distribution ${\mathbb{P}^*}$, $\bar{\rho}$ is the target expected return of the portfolio, and $e$ is the column vector of all ones. It is well known that the minimization of MLSAD is equivalent to the minimization of MAD. MLSAD model can not only inherits the computational advantages of MAD model, but also reflect investors' aversion to downside loss risk. Moreover, MLSAD model has fewer constraints than MAD model.

Since the true distribution $\mathbb{P}^*$ of asset return is unknown, MLSAD model is usually accompanied by uncertainty.  
Distributionally robust optimization (DRO) \cite{Scarf.1958} was proposed as an attractive modeling technique to deal with this uncertainty. Actually, this modeling technique has been widely used in the portfolio model \cite{BJ.Chen.Zhou.2022,Chen.Wu.Li.Ding.Chen.2022,Liu.Yang.Yu.2021,Liu.Chen.A.Xu.2019}. The main idea of distributionally robust optimization  is to find a decision that minimize the  worst-case expected cost over an ambiguity set consisting of distribution information about uncertain data. DRO takes distribution information into account and makes better use of available data, which makes it less conservative than robust optimization  \cite{RH.MS.2019}. The construction of ambiguity set is a key part of distributionally robust optimization model. Distributionally robust optimization with Wasserstein ambiguity set  has received a lot of attention in recent years. As far as we know, Mohajerin Esfahani and Kuhn \cite{EPM.KD.2018} first discovered the convex reformulation  of Wasserstein distributionally robust  optimization model and demonstrated some desirable properties of Wasserstein DRO. Since then, it has been widely used in the field of machine learning and portfolio such as logistic regression \cite{SAS.EPM.KD.2015}, support vector machine \cite{LC.MS.2015,SAS.KD.EPM.2019}, mean-variance model \cite{BJ.Chen.Zhou.2022}, mean-absolute deviation model \cite{Chen.Wu.Li.Ding.Chen.2022}.

The size of  ambiguity set is also a crucial parameter for the performance of  DRO model. A ambiguity set with a suitable size should contain the true distribution with  high confidence, while excluding some pathological distributions. Mohajerin  Esfahani and Kuhn \cite{EPM.KD.2018} presented some methods for selecting the radius of the Wasserstein ball, which includes the $k$-fold cross validation that is popular in practice. $k$-fold cross validation requires solving $k$ optimization problems, we usually consider $k=5$ or $10$. But Shao \cite{SJ.1993} gave a doubt as to whether this choice of $k$ is always appropriate with respect to the sample size $N$.  Blanchet et al.  \cite{BJ.KY.MK.2016} introduced a new projection-based statistical inference method, called {\emph{Robust Wasserstein Profile Inference}, to optimally select the size of  uncertain regions. They also presented the application of this method  to machine learning such as  support vector machines, regularized logistic regression and square-root Lasso. Cisneros-Velarde et al. \cite{CVP.PA.OSY.2020} proposed a  robust selection algorithm following the robust Wasserstein profile inference method to select regularization parameters for DRO graphical Lasso. Hai and Nam \cite{HX.NK.2023} utilized robust Wasserstein profile inference  method to select the appropriate Wasserstein ball radius for distributionally robust mean-CVaR model. However, since the objective function is not differentiable, it is not feasible to apply this method directly to the mean-CVaR model. Then the authors \cite{HX.NK.2023} used the modified method introduced by Peng and Lin  \cite{JMP.ZL.1999} to deal with the objective function.
	
	In this paper, we are mainly concerned with the choice of ambiguity set size for the Wasserstein distributionally  robust mean lower semi-absolute deviation model and its efficient computation. Chen et al. \cite{Chen.Wu.Li.Ding.Chen.2022} have studied the distributionally robust mean absolute deviation model based on the Wasserstein metric (DR-MAD). Although minimizing the MAD model and minimizing the MLSAD model are equivalent, we find that the problem form of the DR-MLSAD model reformulated is simpler than that of the DR-MAD model, which is beneficial for us to analyze the properties of the model and design the algorithm.  Different from the work of Chen et al. \cite{Chen.Wu.Li.Ding.Chen.2022}, we develop the robust method to estimate the size of ambiguity set and  design efficient algorithms for the reformulated problem. Since the objective function in DR-MLSAD model is not differentiable, it is unreasonable to estimate the size of ambiguity set of DR-MLSAD model by using Robust Wasserstein Profile Inference method directly. In view of the work of Hai and Nam \cite{HX.NK.2023}, we utilize a smooth function to approximate the objective function, then we can  derive the size of the ambiguity set by analyzing the properties of the smooth function and the related results of the Robust Wasserstein Profile Inference method.  At the same time, we rewrite the reformulated problem equivalently as an unconstrained optimization problem whose objective function consists of the $\ell_1$-norm function, the indicator function of the polyhedral convex set, and a smooth function. Inspired by a series of pioneering work \cite{XL.DE.KCT.2020,LXD.SDF.2020,LST.2018,ML.DS.KCT.2019,FLX.2021}, we find  high efficiency of semismooth Newton augmented Lagrange algorithm \cite{XL.DE.KCT.2020,LXD.SDF.2020,LST.2018} and semismooth Newton proximal point algorithm \cite{ML.DS.KCT.2019,FLX.2021} in dealing with such problems. This motivates us to design this type of second-order algorithm for the reformulated problem.

	We summarize our main contributions in this paper as follows.
	\begin{itemize}
		\item [1.] We consider a  distributionally robust mean-lower semi-absolute deviation portfolio model under Wasserstein ambiguity set, which can be  reformulated as an equivalent computationally tractable problem. The size of Wasserstein ambiguity set affects the performance of the model. We adopt and modify the {\emph{Robust Wasserstein profile inference}} (RWPI) approach to select the appropriate radius for the DR-MLSAD model. Furthermore, we  compare it with the popular radius selection strategy.
		
		\item [2.] We develop an efficient proximal point dual semismooth Newton algorithm for the convex reformulation of the problem, denoted as {\sc PpdSsn}. The algorithm has been  proved to have global and local convergence under mild conditions. We apply the semismooth Newton method to solve the subproblem in proximal point algorithm and reduce the computational cost by fully exploring the second-order special structure of the problem. 
		
		\item [3.] We  conduct numerical experiments to verify the performance of our proposed model and the efficiency of the {\sc PpdSsn} algorithm. We evaluate the out-of-sample performance of our proposed model by some evaluation criteria and compare it with the  DR-MLSAD model with the radius selected by cross-validation, $1/N$ strategy and SAA model. We also verify the  good performance of {\sc PpdSsn} algorithm by comparing it with first-order algorithms pADMM and dADMM for large-scale problems.
	\end{itemize}
	
	The remaining parts of the paper are organized as follows.  Section \ref{sec:2} is devoted to the construction and reformulation of the Wasserstein  distributionally  robust mean-lower semi-absolute deviation model. In Section \ref{sec:3},  we develop the RWPI approach to find the appropriate Wasserstein radius. We design an efficient algorithm for large-scale convex reformulation optimization problems in Section \ref{sec:4}. Section \ref{sec:5} is dedicated to presenting the out-of-sample performance of our proposed model  and  the performance of {\sc PpdSsn} algorithm on numerical experiments. We give the conclusions in Section \ref{sec:6}.

	{\bf{Notation and preliminaries:}} The following notations and terminologies are used throughout the paper. For given positive integer $m$, we denote $I_{m}$  as the $m\times m$ identity matrix. For given $x\in\mathbb{R}^m$, $``|x|"$ is the absolute vector whose $i$-th entry is $|x_i|$,  $``{\rm{sign}}(x)$" is the sign vector whose $i$-th entry is $1$ if $x_i>0$, $-1$ if $x_i<0$, and $0$ otherwise. For any self-adjoint positive semidefinite linear operator $\mathcal{M}:\mathbb{R}^m\to\mathbb{R}^m$, we define $\left\langle x,x'\right\rangle_{\mathcal{M}}:=\left\langle x,\mathcal{M}x'\right\rangle$, and $\|x\|_{\mathcal{M}}:=\sqrt{\left\langle x,x\right\rangle_{\mathcal{M}}}$ for all $x,x'\in\mathbb{R}^m$. The largest and smallest eigenvalues  of $\mathcal{M}$ are denoted as $\lambda_{\max}(\mathcal{M})$ and $\lambda_{\min}(\mathcal{M})$, respectively. Denote $``{\rm{Diag}}(x)"$ as the diagonal matrix whose diagonal is given by vector $x$.  For given subset $\mathcal{C}\subseteq\mathbb{R}^n$, we define the weighted distance of ${{x}}\in\mathbb{R}^n$  to  $\mathcal{C}$ by ${\rm{dist} }_{\mathcal{M}}({{x}},\mathcal{C}):= \inf_{{{x}}'\in\mathcal{C}}\|{{x}}-{{x}}'\|_{\mathcal{M}}$. We use $``;"$ for adjoining vectors in a column.  The $\ell_{\infty}$-norm unit  ball is defined by $\mathcal{B}_{\infty}:=\{{\bf{x}}\in\mathbb{R}^n \mid \|{\bf{x}}\|_{\infty}\leq 1 \}$. The symbol $``\stackrel{D}{\longrightarrow}"$ is used to denote convergence in distribution. $e$ denotes the vector of all ones.
	
	Let $f:\mathbb{R}^m\to \left( -\infty,+\infty\right]$ be a closed and proper convex function.  For given $t>0$, the Moreau-Yosida regularization and the proximal mapping of $f$ at $x$ is defined as 
	\begin{equation*}
		\begin{aligned}
			E_f^{t}({{x}}) & :=\min\limits_{{{y}}\in \mathbb{R}^m}\{f({{y}})+\frac{1}{2t}\|{{y}}-{{x}}\|^2\},\ \forall {{x}}\in \mathbb{R}^m,\\ {\rm{Prox}}_f({{x}}) & :=\mathop{\arg\min}\limits_{{{y}}\in \mathbb{R}^m}\{f({{y}})+\frac{1}{2}\|{{y}}-{{x}}\|^2\},\ \forall {{x}}\in \mathbb{R}^m.
		\end{aligned}
	\end{equation*}
	As studied in \cite{MJJ.1965}, $E_f^t(\cdot)$ is convex, continuously differentiable and its gradient at ${{x}}\in\mathcal{X}$ is $\nabla E_f^t({{x}})=\left( {{x}}-{\rm{Prox}}_{tf}({{x}})\right)/t$. Moreover, it is easy to know from \cite{LC.SC.1997} that $\nabla E_f^t(\cdot)$ and ${\rm{Prox}}_f(\cdot)$ are globally Lipschitz continuous with modulus $1$. To be specific, for given $t>0$, the proximal mapping of $\ell_1$-norm is given by 
	$$
	{\rm{Prox}}_{t\|\cdot\|_1}({{x}})={\rm{sgn}}({{x}})\circ\max\{|{{x}}|-te,{{0}}\},\ \forall {{x}}\in \mathbb{R}^m.
	$$
	If $f$ is the indicator function of a closed convex set $\mathcal{C}$, denoted as $\chi_{\mathcal{C}}$, then proximal mapping of $\chi_{C}(\cdot)$ reduce to the projection $\Pi_{\mathcal{C}}(\cdot)$ onto the set $\mathcal{C}$:
	$$
	\Pi_{\mathcal{C}}(x):=\mathop{\arg\min}\limits_{{{y}}\in \mathcal{C}}\{\frac{1}{2}\|{{y}}-{{x}}\|^2\},\ \forall {{x}}\in \mathbb{R}^m.
	$$
	Moreover, for given $t>0$, let $\mathcal{D}=t\mathcal{B}_{\infty}=\{u\in\mathbb{R}^m\mid \|u\|_{\infty}\leq t\},$ it holds 
	$$
	{\rm{Prox}}_{\chi_{\mathcal{D}}}(x)=\Pi_{t\mathcal{B}_{\infty}}(x) = {\rm{sign}}(x)\circ\min\{|x|,te\}, \forall x\in\mathbb{R}^m.
	$$

	\section{DRO formulation and computational tractability}\label{sec:2}
In this section, we review some preliminaries on the Wasserstein metric and construct distributionally robust mean-lower semi-absolute deviation (DR-MLSAD) model with Wasserstein ambiguity set. Since the DR-MLSAD model is  complex and difficult to solve directly, we should adopt some techniques to reformulate it into a tractable formulation.

Assume that $(\Xi,\|\cdot\|)$ is a separable complete metric space. Let $\mathcal{M}(\Xi)$  and $\mathcal{M}(\Xi\times\Xi)$ be the set of all Borel probability measures  supported on $\Xi$ and $\Xi\times\Xi$, respectively. The definition of the Wasserstein metric is given as follows.
\begin{definition}
	The Wasserstein metric \cite{KL.RG.1958} $d_{W}:\mathcal{M}(\Xi)\times\mathcal{M}(\Xi)\to\mathbb{R}_{+}$ is defined by
	$$
	\begin{aligned}
	&d_{W}(\mathbb{P}_1,\mathbb{P}_2)\\
	&:=\inf\left\{\int_{\Xi\times\Xi}\|\xi_1-\xi_2\|_p\mathbb{Q}(d\xi_1,d\xi_2):\mathbb{Q}(d\xi_1,\Xi) = \mathbb{P}_1(d\xi_1), \mathbb{Q}(\Xi,d\xi_2) = \mathbb{P}_2(d\xi_2)\right\},
	\end{aligned}
$$
	where $\mathbb{Q}\in\mathcal{M}(\Xi\times\Xi)$ is the joint distribution of $\mathbb{P}_1\in\mathcal{M}(\Xi)$ and $\mathbb{P}_2\in\mathcal{M}(\Xi)$, and  $\|\cdot\|_p$ denotes the $p$-norm with  $p\in\left[ 1,+\infty\right] $.
\end{definition}
For ensuring that $d_{W} $ is a real distance and a metric on $\mathcal{M}(\Xi)$, we give the following assumption \cite{AL.GN.2013,WZ.YK.SS.ZY.2020}.
\begin{assumption}\label{Assum1}
	For any distribution $\mathbb{P}\in\mathcal{M}(\Xi)$, it holds that
	$$
	\mathbb{E}_{\mathbb{P}}\left[ \|\xi\|_p\right] =\int_{\Xi}\|\xi\|_p  \ \mathbb{P}(d\xi)< \infty.
	$$
\end{assumption}
As stated in \cite{WZ.YK.SS.ZY.2020}, Assumption \ref{Assum1}  does not sacrifice much modeling power. We give the  definition of Wasserstein ambiguity set below:
\begin{equation}\label{wass}
	\mathcal{U}_{\epsilon}(\widehat{\mathbb{P}}_N)=\left\lbrace \mathbb{P}: \mathbb{P}\left\lbrace \xi\in\mathcal{M}(\Xi)\right\rbrace =1,d_{W}(\mathbb{P},\widehat{\mathbb{P}}_N)\leq\epsilon\right\rbrace,
\end{equation}
where $\widehat{\mathbb{P}}_N$ is the discrete empirical probability distribution, that is,
$$
\widehat{\mathbb{P}}_N=\frac{1}{N}\sum_{i=1}^{N}\delta_{\hat{\xi}_i},
$$
where $\{\hat{\xi}_1,\hat{\xi}_2,\dots,\hat{\xi}_N\}$ is a set of the independent observations of $\xi$ and  $\delta_{\hat{\xi}_i}$ is the Dirac point measure at ${\hat{\xi}_i}$. In fact, the Wasserstein  ambiguity set can be viewed as a Wasserstein ball with $\widehat{\mathbb{P}}_N$ as the center and $\epsilon\geq0$ as the radius. We want to characterize the mean-lower semi-absolute deviation model with the Wasserstein distributionally robust formulation  that minimizes the worst-case semi-absolute deviation while the worst-case returns exceeds a given target value $\rho$, that is,
\begin{equation}\label{DRMLSAD}
	\begin{aligned}
		&\min_{x\in\mathcal{X}}\max_{\mathbb{P}\in\mathcal{U}_{\epsilon}(\widehat{\mathbb{P}}_N)}\ \mathbb{E}_{\mathbb{P}}|\min(0,\xi^{\top}x-\mathbb{E}_{\mathbb{P}}(\xi^{\top}x))|\\
		&\ {\rm{s.t.}} \min_{\mathbb{P}\in\mathcal{U}_{\epsilon}(\widehat{\mathbb{P}}_N)}\ \mathbb{E}_{\mathbb{P}}(\xi^{\top}x)\geq \rho,
	\end{aligned}
\end{equation}
where $\rho$ can be considered as the lowest acceptable target return.  We  take target return $\bar{\rho}$ and the size of the uncertainty set $\epsilon$ into account when we choose $\rho$. It is more appropriate to select $\rho<\bar{\rho}$ such that $\bar{\rho}-\rho$ is informed by $\epsilon$.
In particular, for $\epsilon=0$, the ambiguity set reduces to the singleton $\{\widehat{\mathbb{P}}_N\}$, then problem \eqref{DRMLSAD} reduces to the following problem:
\begin{equation}\label{saa}
	\begin{aligned}
		&\min_{x\in\mathcal{X}}\ \mathbb{E}_{\widehat{\mathbb{P}}_N}\left[ \max(0,\mathbb{E}_{\widehat{\mathbb{P}}_N}(\xi^{\top}x)-\xi^{\top}x)\right] \\
		&\ {\rm{s.t.}}\ \  \mathbb{E}_{\widehat{\mathbb{P}}_N}(\xi^{\top}x)\geq{\rho}.
	\end{aligned}
\end{equation}
In this paper, we regard problem \eqref{saa} as a sample average approximation of problem \eqref{MLSAD}, called SAA problem.

We reformulate the DR-MLSAD model \eqref{DRMLSAD} in a computationally tractable way.  
Denote the  feasible region as $\mathcal{F}_{\epsilon,\rho} =\{x\in\mathbb{R}^m\mid e^{\top}x=1,x\geq0,\min_{\mathbb{P}\in\mathcal{U}_{\epsilon}(\widehat{\mathbb{P}}_N)}\ \mathbb{E}_{\mathbb{P}}(\xi^{\top}x)\geq \rho\}.$ Then the problem \eqref{DRMLSAD} can be rewritten as
$$
\min_{x\in\mathcal{F}_{\epsilon,\rho}}\max_{\mathbb{P}\in\mathcal{U}_{\epsilon}(\widehat{\mathbb{P}}_N)}\ \mathbb{E}_{\mathbb{P}}\left[ \max(0,\mathbb{E}_{\mathbb{P}}(\xi^{\top}x)-\xi^{\top}x)\right] .
$$
By virtue of  the formulation of model \eqref{DRMLSAD}, we find that model \eqref{DRMLSAD} contains an infinite number of constraints due to the infinite number of elements in the Wassertstein ambiguity set $\mathcal{U}_{\epsilon}(\widehat{\mathbb{P}}_N)$, which leads to the model being intractable. Moreover, it is clear to see that  problem \eqref{DRMLSAD} is non-convex. Fortunately, referring to the seminal work \cite{EPM.KD.2018,Chen.Wu.Li.Ding.Chen.2022,BJ.Chen.Zhou.2022}, we can reformulate the model \eqref{DRMLSAD} as a finite convex optimization problem. By fixing $\mathbb{E}_\mathbb{P}(\xi^{\top}x)=\alpha\geq\rho$ in the objective function of problem \eqref{DRMLSAD}, we obtain the following equivalent form:
\begin{equation}\label{EDRMLSAD}
	\min_{ x\in\mathcal{F}_{\epsilon,\rho} }\left\{\max_{\alpha\geq \rho}\left[\max_{\mathbb{P}\in\mathcal{U}_{\epsilon}(\widehat{\mathbb{P}}_N),\mathbb{E}_{\mathbb{P}}(\xi^{\top}x)=\alpha}\ \mathbb{E}_{\mathbb{P}}\left[ \max(0,\alpha-\xi^{\top}x)\right]  \right]\right\}.
\end{equation}
It is obvious that problem \eqref{EDRMLSAD} is a {\emph{min-max-max}} problem. Now, we describe the process of reformulating problem \eqref{EDRMLSAD}:
First of all, the inner maximization problem is over a collection of probability measures, which is computationally intractable. We reformulate it as a computationally tractable problem by some duality theory and obtain its optimal value. Next, we take the optimal value of the inner maximization problem into the second-stage maximization problem in problem \eqref{EDRMLSAD}. The variable in the second-stage maximization problem is one-dimensional,  we can also easily find the optimal value of the second-stage problem. Finally, we plug the optimal value from the second-stage into the outer minimization problem. In addition, the feasible region $\mathcal{F}_{\epsilon,\rho}$ contains an infinite dimensional optimization problem, thus we also transform the feasible region of the outer minimization problem into a tractable formulation. The relevant result is given in the following theorem.
\begin{theorem}\label{thm2}
	Suppose $\|\cdot\|_p=\|\cdot\|_{\infty}$. Considering the uncertain set $\Xi=\mathbb{R}^m$. Then for any $\epsilon\geq0$, DR-MLSAD model is equivalent to the following problem:
	\begin{equation}\label{DR2}
		\begin{aligned}
			&\min_{x\in\mathcal{X}}\ \frac{1}{2N}\sum_{i=1}^{N}|(\hat{\mu}-\hat{\xi}_i)^{\top}x-\epsilon|+\frac{1}{2N}\sum_{i=1}^{N}((\hat{\mu}-\hat{\xi}_i)^{\top}x-\epsilon)+\epsilon\\
			\ & \ {\rm{s.t.}} \quad \hat{\mu}^{\top}x\geq \rho+\epsilon,
		\end{aligned}
	\end{equation}
	where $\hat{\mu}=\frac{1}{N}\sum_{i=1}^{N}\hat{\xi}_i$.
\end{theorem}
\begin{proof}
See Appendix \ref{App1}.
\end{proof}


\section{Choice of ambiguity radius $\epsilon$}\label{sec:3}
In this section, we focus on the choice of ambiguity radius $\epsilon$ such that the DR-MLSAD model has some desirable statistical properties.  We adopt a novel approach,  \emph{Robust Wasserstein Profile Inference} (RWPI) proposed by Blanchet et al. \cite{BJ.KY.MK.2016}, to estimate  the suitable  radius. 

Recalling problem \eqref{MLSAD}: 
\begin{equation}\label{primalpp}
	\min_{x\in\mathbb{R}^m}\ \mathbb{E}_{\mathbb{P}^*}\left\lbrace \max(0,\mathbb{E}_{\mathbb{P}^*}(\xi^{\top}x)-\xi^{\top}x)\right\rbrace \quad
	{\rm{s.t.}}\ \  \mathbb{E}_{\mathbb{P}^*}(\xi^{\top}x)=\bar{\rho}, e^{\top}x = 1,x\geq0,
\end{equation} 
we find that the objective function is not differentiable. Therefore, it is not appropriate to directly apply the RWPI approach to estimate the radius. Before adopting the RWPI method, we first find a smooth function to approximate the objective function $\max(0,\bar{\rho}-\xi^{\top}x)$, then we use this smooth function to carry out the later analysis. Referring to the work  \cite{JMP.ZL.1999, LXS.1991}, we approximate the \emph{max} function by the following smooth function:
$$
f_t(x) = t\ln\left( 1+\exp\left( \frac{\bar{\rho}-\xi^{\top}x}{t}\right)\right) .
$$
This function  is continuously differentiable with respect to $x$ and  $f_t(x)\to\max(0,\bar{\rho}-\xi^{\top}x)$ as $t\to 0$. Thus, we consider the following  problem
\begin{equation}\label{smooth1}
	\min_{x\in\mathbb{R}^m}\ \mathbb{E}_{\mathbb{P}^*}\left\lbrace t\ln\left( 1+\exp\left( \frac{\bar{\rho}-\xi^{\top}x}{t}\right)\right) \right\rbrace \quad
	{\rm{s.t.}}\ \ \mathbb{E}_{\mathbb{P}^*}(\xi^{\top}x)=\bar{\rho}, e^{\top}x=1,x\geq0.
\end{equation}
The Lagrangian function of problem \eqref{smooth1} takes the following form:
\begin{equation}\nonumber
	\mathcal{L}(x;\lambda)=\mathbb{E}_{\mathbb{P}^*} \left[ t\ln\left( 1+\exp\left( \frac{\bar{\rho}-\xi^{\top}x}{t}\right)\right)\right]  +\lambda_1(\bar{\rho}-\mathbb{E}_{\mathbb{P}^*}(\xi^{\top}x))+\lambda_2(1-e^{\top}x)-\left\langle \lambda_3,x\right\rangle,
\end{equation}
where $\lambda_1,\lambda_2\in\mathbb{R},\lambda_3\in\mathbb{R}^m$ are Lagrange multipliers.
Given  target value $\bar{\rho}$, we assume that  problem \eqref{smooth1} has a unique solution, denoted by $x^*.$ 
Thus, there exists Lagrange multiplier $\left( \lambda_1^*,\lambda_2^*,\lambda_3^*\right) $ such that 
\begin{align}
	&D_{x^*}\mathcal{L}(x^*;\lambda_1^*,\lambda_2^*,\lambda_3^*) = \mathbb{E}_{\mathbb{P}^*}\left[ \frac{-\xi}{1+\exp((\xi^{\top}x^*-\bar{\rho})/t)}\right]   -\lambda_1^*\mathbb{E}_{\mathbb{P}^*}(\xi)-\lambda_2^*e-\lambda_3^*=0,\label{kktPopt1}\\
	& \mathbb{E}_{\mathbb{P}^*}(\xi^{\top}x^*)-\bar{\rho}=0,\ e^{\top}x^*-1=0,\ 0\leq \lambda_3^*\perp x^*\geq0. \label{kktPopt2}
\end{align}
Define 
$$
g(x^*,\xi):=\frac{-\xi}{1+\exp((\xi^{\top}x^*-\bar{\rho})/t)}.
$$ 
Multiplying \eqref{kktPopt1} by $(x^*)^{\top}$ and invoking \eqref{kktPopt2}, we have
\begin{equation}\label{lambda2}
	\lambda_2^* = (x^*)^{\top} \mathbb{E}_{\mathbb{P}^*}\left[  g(x^*,\xi)\right]  -\lambda_1^*\bar{\rho}.
\end{equation}
Substituting $\lambda_2^*$ into equation \eqref{kktPopt1}, we obtain
\begin{equation}\label{lambda1}
	\lambda_1^* = \frac{\mathbb{E}_{\mathbb{P}^*}\left[ g^i(x^*,\xi)\right]- (x^*)^{\top} \mathbb{E}_{\mathbb{P}^*}\left[  g(x^*,\xi)\right]}{\mathbb{E}_{\mathbb{P}^*}(\xi^i)-\bar{\rho}},\ (\lambda_3^*)^i=0,
\end{equation}
where $ i\in \mathcal{I}_*:=\left\lbrace i\in\left\lbrace 1,\ldots,m\right\rbrace \mid (x^*)^i \neq0\right\rbrace.$
On the other hand, for $ i\in \bar{\mathcal{I}}_*:=\left\lbrace i\in\left\lbrace 1,\ldots,m\right\rbrace \mid (x^*)^i =0\right\rbrace,$ we get 
\begin{equation}\label{lambda3}
	(\lambda_3^*)^i=\mathbb{E}_{\mathbb{P}^*}\left[ g^i(x^*,\xi)\right] -\lambda_1^*\mathbb{E}_{\mathbb{P}^*}(\xi^i)-\lambda_2^*.
\end{equation}
Denote $h(\xi;x^*)  = g(x^*,\xi)-\lambda_1^*\xi-\lambda_2^*e-\lambda_3^*.$ Thus, the first-order optimality condition for problem \eqref{smooth1} can be written as
$$
\mathbb{E}_{\mathbb{P}^*}\left[ h(\xi;x^*)\right] =0.
$$

Note that $\epsilon$ should be suitably selected  such that the set $
\mathcal{U}_{\epsilon}(\widehat{\mathbb{P}}_N)=\lbrace \mathbb{P}: d_{W}(\mathbb{P},\widehat{\mathbb{P}}_N)\leq\epsilon\rbrace$ contains all the probability measures that are plausible variations represented by $\widehat{\mathbb{P}}_N$. Since any 
$$
x_{\mathbb{P}}=\mathop{\arg\min}\limits_{x\in\mathbb{R}^m}\left\lbrace \mathbb{E}_{\mathbb{P}}\left[  t\ln\left( 1+\exp\left( \frac{\bar{\rho}-\xi^{\top}x}{t}\right)\right)\right]  \bigg|\  \mathbb{E}_{\mathbb{P}}(\xi^{\top}x)=\bar{\rho}, x\in\mathcal{X}\right\rbrace
$$
with $\mathbb{P}\in\mathcal{U}_{\epsilon}(\widehat{\mathbb{P}}_N)$ is a plausible estimate of $x^*$,
we define the set
$$
\Lambda_{\epsilon}(\widehat{\mathbb{P}}_N):=\bigcup_{\mathbb{P}\in\mathcal{U}_{\epsilon}(\widehat{\mathbb{P}}_N)}x_{\mathbb{P}},
$$
which  contains  all plausible selections of $x^*$. 
The set $\Lambda_{\epsilon}(\widehat{\mathbb{P}}_N)$ is a natural confidence region for $x^*$ when $\epsilon$ is sufficiently large. Our goal is  to find the smallest $\epsilon^*_N$ such that $x^*\in\Lambda_{\epsilon}(\widehat{\mathbb{P}}_N)$  holds with a given confidence level $1-\bar{\alpha}$, i.e.,
\begin{equation}\label{optepsi}
	\epsilon^*_N = \min\left\lbrace \epsilon>0: \mathbb{P}^*\left( x^*\in\Lambda_{\epsilon}(\widehat{\mathbb{P}}_N)\right) \geq1-\bar{\alpha}\right\rbrace.
\end{equation}
Since problem \eqref{optepsi} is difficult to solve, we refer to the RWPI approach to solve problem \eqref{optepsi} asymptotically. It is not hard to find that $x\in\Lambda_{\epsilon}(\widehat{\mathbb{P}}_N)$ if and only if there exist $\mathbb{P}\in\mathcal{U}_{\epsilon}(\widehat{\mathbb{P}}_N)$, $\lambda_1,\lambda_2\in\mathbb{R},\lambda_3\in\mathbb{R}_{+}^m$ such that the first order optimality condition holds, i.e.,
\begin{equation}\label{kktP}
	\mathbb{E}_{\mathbb{P}}\left[ h(\xi;x)\right] = \mathbb{E}_{\mathbb{P}}\left[g(x,\xi)-\lambda_1\xi-\lambda_2e-\lambda_3\right]=0.
\end{equation}
Robust Wasserstein  profile (RWP) function associated with the conditions \eqref{kktP} is defined by 
\begin{equation}\label{RWPfun}
	R_N(x):=\min\left\lbrace d_W(\mathbb{P},\widehat{\mathbb{P}}_N):
	\mathbb{E}_{\mathbb{P}}\left[ h(\xi;x)\right] =0\right\rbrace.
\end{equation}
From the definition of $R_N(\cdot)$ and $h(\cdot;\cdot)$, we have
$$
R_N(x^*):=\min\left\lbrace d_W(\mathbb{P},\widehat{\mathbb{P}}_N):
\mathbb{E}_{\mathbb{P}}\left[ h(\xi;x^*)\right] =\mathbb{E}_{\mathbb{P}}\left[ g(x^*,\xi)-\lambda_1^*\xi-\lambda_2^*e-\lambda_3^*\right] =0\right\rbrace,
$$
where $\lambda_1^*,\lambda_2^*$ and $\lambda_3^*$ are the Lagrange multipliers in \eqref{kktPopt1} and \eqref{kktPopt2}.
Thus, the following relation holds:
$$
R_N(x^*)\leq\epsilon\iff x^*\in\Lambda_{\epsilon}(\widehat{\mathbb{P}}_N).
$$
Therefore, problem \eqref{optepsi} is equivalent to
$$
\epsilon^*_N = \min\left\lbrace \epsilon>0: \mathbb{P}^*\left( R_N(x^*)\leq \epsilon \right) \geq1-\bar{\alpha}\right\rbrace.
$$

By virtue of \cite[Proposition 4]{BJ.KY.MK.2016} and \cite[Theorem 3.3]{HX.NK.2023}, one can obtain the weak convergence result with respect to $R_N(x^*)$, which is shown in the next theorem.
\begin{theorem}\label{thm3.1}
	Assume that  problem \eqref{smooth1} has a unique solution $x^*$. As $N\to\infty$, we have 
	$$
	N^{1/2}R_N(x^*)\stackrel{D}{\longrightarrow}\bar{R}(1),
	$$
	where 
	$$
	\begin{aligned}
	&\bar{R}(1):=\max_{\zeta\in\bar{\mathbb{E}}}\left\lbrace\zeta^{\top}H  \right\rbrace \ {\rm{with}}\  \bar{\mathbb{E}}=\left\lbrace\zeta\in\mathbb{R}^N:\|\zeta^{\top}
	 D_{\xi}h(\xi,x^*)\|_1 \leq 1\right\rbrace,\\
	&H\sim\mathcal{N}\left( {\bf{0}}, Cov_{\mathbb{P}^*}\left[h(\xi;x^*) \right] \right)
	\end{aligned}
	$$
	and $Cov_{\mathbb{P}^*}\left[h(\xi;x^*) \right] = \mathbb{E}_{\mathbb{P}^*}\left[ h(\xi;x^*)h(\xi;x^*)^{\top}\right]$,
	where 
	$$
	\begin{aligned}
	D_{\xi}h(\xi,x^*) &= \frac{x^*\xi ^{\top}}{t\left[ 1+\exp((\xi^{\top}x^*-\bar{\rho})/t)\right] \left[ 1+\exp((-\xi^{\top}x^*+\bar{\rho})/t)\right] }\\
	&+\frac{-I_m}{1+\exp((\xi^{\top}x^*-\bar{\rho})/t)}-\lambda_1^*I_m.
	\end{aligned}
	$$
	Furthermore, $	\bar{R}(1)$ has the stochastic upper bound:
	$$
	\bar{R}(1)=\max_{\zeta\in\bar{\mathbb{E}}}\left\lbrace\zeta^{\top}H  \right\rbrace\overset{D}{\leq} \max_{\|\zeta\|_1\leq 1}\left\lbrace\zeta^{\top}H  \right\rbrace:=\|\tilde{H}\|_\infty,
	$$
	where
	$$
	\tilde{H}\sim\mathcal{N}\left( {\bf{0}}, Cov_{\mathbb{P}^*}\left[(1+|\hat{\lambda}_1^*|)|\xi|+|\hat{\lambda}_2^*|e+\hat{\lambda}_3^* \right] \right),
	$$ 
	$	\bar{R}(1)\overset{D}{\leq}\|\tilde{H}\|_\infty$ represents that $\bar{R}(1)$ is stochastically dominated by $\|\tilde{H}\|_\infty$, that is, for $a\in\mathbb{R}$, $\mathbb{P}(\bar{R}(1)\leq a)\geq \mathbb{P}(\|\tilde{H}\|_\infty\leq a) $ holds
	and 
	\begin{equation}\label{lambda123}
		\left\{
		\begin{aligned}
			&	\hat{\lambda}_1^*=\lim\limits_{t\to0^{+}}\lambda_1^*=\frac{-\mathbb{E}_{\mathbb{P}^*}\left( \xi^i{\bf{1}}_{\left\lbrace \xi^{\top}x^*-\bar{\rho}<0\right\rbrace }\right) +(x^*)^{\top}\mathbb{E}_{\mathbb{P}^*}\left( \xi{\bf{1}}_{\left\lbrace \xi^{\top}x^*-\bar{\rho}<0\right\rbrace }\right)}{\mathbb{E}_{\mathbb{P}^*}(\xi^i)-\bar{\rho}},\  i\in \mathcal{I}_*,\\
			&\hat{\lambda}_2^*=\lim\limits_{t\to0^{+}} \lambda_2^*= (x^*)^{\top}\mathbb{E}_{\mathbb{P}^*}\left( \xi{\bf{1}}_{\left\lbrace \xi^{\top}x^*-\bar{\rho}<0\right\rbrace }\right)  -\hat{\lambda}_1^*\bar{\rho},\\
			&\hat{\lambda}_3^*=\lim\limits_{t\to0^{+}} \lambda_3^*=\mathbb{E}_{\mathbb{P}^*}\left( \xi{\bf{1}}_{\left\lbrace \xi^{\top}x^*-\bar{\rho}<0\right\rbrace }\right) -\hat{\lambda}_1^*\mathbb{E}_{\mathbb{P}^*}(\xi)-\hat{\lambda}_2^*e.
		\end{aligned}\right.
	\end{equation}
\end{theorem}
\begin{proof}
		See Appendix \ref{App2}.
\end{proof}

Referring to \cite{BJ.MK.ANV.2021}, we  propose the robust selection approach for estimating $\epsilon$, which is presented in the following Algorithm \ref{algepsi}.
\begin{algorithm}[H]
	\caption{ An approach  for estimating  the Wasserstein radius $\epsilon$ in \eqref{DRMLSAD}} \label{algepsi}
	\begin{algorithmic}
		\Require  Training samples $\hat{\xi}_1,\dots,\hat{\xi}_N$. Target return $\bar{\rho}.$
		\Ensure The appropriate radius $\epsilon$.
		\State  {\bf{Step 1:}} Compute 
		$$
		\begin{aligned}
			&	\hat{\lambda}_1^{\rm{erm}}=\frac{-\mathbb{E}_{\widehat{\mathbb{P}}_N}\left( (\hat{\xi})^i{\bf{1}}_{\left\lbrace \hat{\xi}^{\top}\hat{x}^{\rm{erm}}-\bar{\rho}<0\right\rbrace }\right) +(\hat{x}^{\rm{erm}})^{\top}\mathbb{E}_{\widehat{\mathbb{P}}_N}\left( \hat{\xi}{\bf{1}}_{\left\lbrace \hat{\xi}^{\top}\hat{x}^{\rm{erm}}-\bar{\rho}<0\right\rbrace }\right)}{\mathbb{E}_{\widehat{\mathbb{P}}_N}\left[ (\hat{\xi})^i\right] -\bar{\rho}},\\
			&  i\in \left\lbrace i\mid (\hat{x}^{\rm{erm}})^i\neq 0\right\rbrace,\\
			&\hat{\lambda}_2^{\rm{erm}}= (\hat{x}^{\rm{erm}})^{\top}\mathbb{E}_{\widehat{\mathbb{P}}_N}\left( \hat{\xi}{\bf{1}}_{\left\lbrace \hat{\xi}^{\top}\hat{x}^{\rm{erm}}-\bar{\rho}<0\right\rbrace }\right)  -\hat{\lambda}_1^{\rm{erm}}\bar{\rho},\\
			&\hat{\lambda}_3^{\rm{erm}}=\mathbb{E}_{\widehat{\mathbb{P}}_N}\left( \hat{\xi}{\bf{1}}_{\left\lbrace \hat{\xi}^{\top}\hat{x}^{\rm{erm}}-\bar{\rho}<0\right\rbrace }\right) -\hat{\lambda}_1^{\rm{erm}}\mathbb{E}_{\widehat{\mathbb{P}}_N}(\hat{\xi})-\hat{\lambda}_2^{\rm{erm}}e
		\end{aligned}
		$$
		to approximate $\hat{\lambda}_1^*,\hat{\lambda}_2^*,\hat{\lambda}_3^*$ in \eqref{lambda123},
		where $\hat{x}_N^{\rm{erm}}$ is the optimal solution of the problem \eqref{MLSAD} under distribution $\widehat{\mathbb{P}}_N.$
		
		\State  {\bf{Step 2:}} Denote  $\tilde{h}(\xi,\lambda_1,\lambda_2,\lambda_3):=(1+|{\lambda}_1|)|\xi|+|{\lambda}_2|e+{\lambda}_3$. Compute the covariance matrix $\hat{\Sigma}$ of $\left\lbrace \tilde{h}(\hat{\xi}_1,\hat{\lambda}_1^{\rm{erm}},\hat{\lambda}_2^{\rm{erm}},\hat{\lambda}_3^{\rm{erm}}),\dots,\tilde{h}(\hat{\xi}_N,\hat{\lambda}_1^{\rm{erm}},\hat{\lambda}_2^{\rm{erm}},\hat{\lambda}_3^{\rm{erm}})\right\rbrace.
		$
		\State {\bf{Step 3:}} Obtain independent samples $\tilde{H}_1,\dots,\tilde{H}_k$ from $\mathcal{N}({\bf{0}},\hat{\Sigma})$, where $k = \lfloor 0.2N \rfloor.$ Let $\eta_{1-\bar{\alpha}}$ be the $(1-\bar{\alpha})$-quantile  of $\|\tilde{H}\|_{\infty}$.
		\State {\bf{Step 4:}} Return $\epsilon = \sqrt{N}\times\eta_{1-\bar{\alpha}}$.
	\end{algorithmic}
\end{algorithm}

\section{A preconditioned {\sc{PpdSsn}} algorithm for DR-MLSAD model}\label{sec:4}
In this section, we aim to design an efficient algorithm for solving problem \eqref{DR2}, which is equivalent to the DR-MLSAD model. 

Define the set $ \mathcal{C}:=\lbrace x\in\mathbb{R}^m\mid\hat{\mu}^{\top}x\geq \rho+\epsilon,x\in\mathcal{X}\rbrace$, where $\mathcal{X}$ denotes the simplex. Denote the matrix $A=[\hat{\mu}-\hat{\xi}_1,\ldots,\hat{\mu}-\hat{\xi}_N] ^{\top}\in\mathbb{R}^{N\times m}$. Then problem \eqref{DR2} can be rewritten as the following unconstrained convex problem:
\begin{equation}\label{3.8}
	\min_{x\in\mathbb{R}^m}\ \left\{\frac{1}{2N}\|Ax-\epsilon e\|_1+\frac{1}{2N}e^{\top}(Ax-\epsilon e)+\epsilon+\chi_{\mathcal{C}}(x)\right\},
\end{equation}
where $\chi_{\mathcal{C}}(\cdot)$ is the indicator function of  $\mathcal{C}$.   Problem \eqref{3.8} can be equivalently expressed as
\begin{equation}\label{DR3}
	\begin{aligned}
		&\min_{x\in\mathbb{R}^m,y\in\mathbb{R}^N}\ \frac{1}{2N}e^{\top}y+ \frac{1}{2N}\|y\|_1+\epsilon+\chi_{\mathcal{C}}(x)\\
		\quad& \quad \quad {\rm{s.t.}} \quad\ y=Ax-\epsilon e.
	\end{aligned}
\end{equation}
The Karush-Kuhn-Tucker (KKT)  conditions for problem \eqref{DR3} are formulated as follows:
\begin{equation}\label{kkt}
	\left\{\begin{aligned}
		&-\frac{1}{2N}e-u\in\frac{1}{2N}\partial \|y\|_1,\\ &A^{\top}u\in\partial\chi_{\mathcal{C}}(x),\\
		&y = Ax-\epsilon e,
	\end{aligned}
	\right.
\end{equation}
where $u\in\mathbb{R}^N$ is the Lagrange multiplier. Denote the optimal solution set of problem \eqref{DR3} by $\mathcal{S}_{P}$ and the KKT residual function corresponding to problem \eqref{DR3} by
$$
R(x,y):= \left [                
\begin{array}{c}    
	y-{\rm{Prox}}_{\frac{1}{2N}\|\cdot\|_1}(y-\frac{1}{2N}e-u)\\
	x-\Pi_{\mathcal{C}}(x+A^{\top}u) \\
	y-Ax+\epsilon e
\end{array}
\right],
$$
which means that $(x^*,y^*)\in\mathcal{S}_{P}$ if and only if  $R(x^*,y^*)=0.$ We assume that the optimal solution set $\mathcal{S}_{P}$ is  nonempty and compact, which is reasonable due to \cite{ZZ.AMS.2017}. 

Next, we apply a preconditioned proximal point dual semismooth Newton ({\sc{PpdSsn}}) algorithm for  problem \eqref{DR3}. Given a nondecreasing sequence of positive real numbers $\left\{\sigma_k\right\}$ and a starting point $x^0$,  the preconditioned {\sc Ppa} \cite{RRT.1976} generates the sequence $\left\{x^k\right\}$  via the following rules for solving \eqref{DR3}:
\begin{equation}\label{ppa}
	\min_{x\in\mathbb{R}^m}\left\{
	\frac{1}{2N}e^{\top}(Ax-\epsilon e)+ \frac{1}{2N}\|Ax-\epsilon e\|_1+\epsilon+\chi_{\mathcal{C}}(x)+\frac{1}{2\sigma_k}\|x-x^k\|^2_{\widetilde{\mathcal{M}}_k}
	\right\},
\end{equation}
where $\widetilde{\mathcal{M}}_k$ is a given sequence of self-adjoint positive definite linear operators that satisfy
$$
\widetilde{\mathcal{M}}_{k}\succeq{\widetilde{\mathcal{M}}}_{k+1},\ \widetilde{\mathcal{M}}_k\succeq\lambda_{\min}(\widetilde{\mathcal{M}}_{k})I_{m},\ \forall k\geq0.
$$
In this paper, we choose $\widetilde{\mathcal{M}}_k=I_{m}+\left( \sigma_k/\gamma_k\right) A^{\top}A$, where $\{\gamma_k\}$ is a given sequence of positive real numbers.
Specifically, the problem \eqref{ppa} is equivalent to the following problem:
\begin{equation}\label{eppa}
	\begin{aligned}
		&(x^{k+1};y^{k+1})\approx\mathcal{P}_k(x^k,y^k)\\
		&:=\mathop{\arg\min}\limits_{x,y}\left\{
		\begin{aligned}
			&\frac{1}{2N}e^{\top}y+ \frac{1}{2N}\|y\|_1+\epsilon+\chi_{\mathcal{C}}(x)\\
			&+\frac{1}{2\sigma_k}\|x-x^k\|^2
			+\frac{1}{2\gamma_k}\|y-(Ax^k-\epsilon e)\|^2
		\end{aligned}\ \Bigg| \ y = Ax-\epsilon e
		\right\}.
	\end{aligned}
\end{equation}
The corresponding Lagrangian function of problem \eqref{eppa} is given by
$$
\begin{aligned}
L(x,y;u)&=\frac{1}{2N}e^{\top}y+ \frac{1}{2N}\|y\|_1+\epsilon+\chi_{\mathcal{C}}(x)+\frac{1}{2\sigma_k}\|x-x^k\|^2\\
&+\frac{1}{2\gamma_k}\|y-(Ax^k-\epsilon e)\|^2+\left\langle u,y-Ax+\epsilon e\right\rangle.
\end{aligned}
$$
The dual of problem \eqref{eppa} admits the following minimization form:
\begin{equation}\label{ppadual}
	\min_{\lambda\in\mathbb{R}^N}\left\{
	\begin{aligned}
		\psi_k(\lambda):=&-E^{\sigma_k}_{\delta_{\mathcal{C}}}(x^k+\sigma_kA^{\top}u)-\frac{1}{2\sigma_k}\|x^k\|^2+\frac{1}{2\sigma_k}\|x^k+\sigma_kA^{\top}u\|^2\\
		&-\frac{1}{\gamma_k}E^1_{\frac{\gamma_k}{2N}\|\cdot\|_1}(Ax^k-\epsilon e-\gamma_k(\frac{1}{2N}e+u))-\frac{1}{2\gamma_k}\|Ax^k-\epsilon e\|^2\\
		&+\frac{1}{2\gamma_k}\|Ax^k-\epsilon e-\gamma_k(\frac{1}{2N}e+u)\|^2-\epsilon(1+u^{\top} e) 
	\end{aligned}
	\right\}.
\end{equation}
Indeed, one can easily derive that if $u^{k+1}\approx\mathop{\arg\min}\limits_{\lambda\in\mathbb{R}^N}\psi_k(u),$ then $x^{k+1}$ and $y^{k+1}$ admit the following closed-form expression:
$$
x^{k+1}=\Pi_{\mathcal{C}}(x^k+\sigma_kA^{\top}u^{k+1}),\ y^{k+1}={\rm{Prox}}_{\frac{\gamma_k}{2N}\|\cdot\|_1}(Ax^k-\epsilon e-\gamma_k(\frac{1}{2N}e+u^{k+1})).
$$
It should be noted that  $\Pi_{\mathcal{C}}(\cdot)$ can be computed by the semismooth Newton ({\sc Ssn}) algorithm or an algorithm based on Lagrangian
relaxation approach and secant method (LRSA) (cf. \cite{YJL.WMZ.2023}). In our experiment, we apply the more efficient LRSA algorithm to calculate $\Pi_{\mathcal{C}}(\cdot)$.

\subsection{ The {\sc PpdSsn} algorithm and its convergence}
We give the specific framework of the {\sc PpdSsn} algorithm in Algorithm \ref{Alg:1} and present the result on convergence of the preconditioned {\sc PpdSsn} algorithm.
\begin{algorithm}[H]
	\caption{  ({\sc{PpdSsn}}) Preconditioned proximal point dual semismooth Newton algorithm for \eqref{DR3}} \label{Alg:1}
	\begin{algorithmic}[1]
		\Require  $\sigma_0>0$, $\gamma_0>0$, $(x^0,y^0)\in\mathbb{R}^m\times\mathbb{R}^N , u^0\in\mathbb{R}^N$. Set $k=0$. 
		\State  Approximately compute
		\begin{equation}\label{sub}
			u^{k+1}\approx\mathop{\arg\min}\limits_{u\in\mathbb{R}^N}\psi_k(u)
		\end{equation}
		to satisfy the stopping criteria \eqref{eqn1} and \eqref{eqn2}.
		\State Compute 
		$$
		x^{k+1}=\Pi_{\mathcal{C}}(x^k+\sigma_kA^{\top}u^{k+1}),\ y^{k+1}={\rm{Prox}}_{\frac{\gamma_k}{2N}\|\cdot\|_1}(Ax^k-\epsilon e-\gamma_k(\frac{1}{2N}e+u^{k+1})).
		$$
		
		\State Update $\sigma_{k+1} \uparrow \sigma_\infty<\infty,\gamma_{k+1} \uparrow \gamma_\infty< \infty$, $k\leftarrow k+1,$  go to step $1$.
	\end{algorithmic}
\end{algorithm} 
For later analysis, we define the following function $f_k$:
\begin{equation}\label{fk}
	f_k(x,y):=\frac{1}{2N}e^{\top}y+ \frac{1}{2N}\|y\|_1+\epsilon+\chi_{\mathcal{C}}(x)+\frac{1}{2\sigma_k}\left\| 
	(x;y)
	-(x^k;Ax^k-\epsilon e)\right\|^2_{{\mathcal{M}}_k}+\chi_{\widetilde{\mathcal{F}}}(x,y),
\end{equation}
where ${\mathcal{M}}_k={\rm{Diag}}(e;\sigma_k{\gamma_k}^{-1}e)\in\mathbb{R}^{(m+N)\times(m+N)}$ and $\widetilde{\mathcal{F}}$  denotes the feasible set of \eqref{eppa}.
Thus, problem \eqref{ppa} can be written as follows:
$$
(x^{k+1};y^{k+1})\approx\mathcal{P}_k(x^k,y^{k+1}):=\mathop{\arg\min}\limits_{x\in\mathbb{R}^m,y\in\mathbb{R}^N}\left\{
f_k(x,y)
\right\}.
$$

Before analyzing the convergence of the {\sc{PpdSsn}} algorithm, it is necessary to make the following assumption \cite{XL.DE.KCT.2020}.
\begin{assumption}\label{assum1}
	The sequences $\left\{\sigma_k\gamma_k^{-1}\right\}$ and $\left\{\sigma_k\right\}$ of positive real numbers are both bounded
	away from 0, i.e., $(\sigma_k\gamma_k^{-1})\downarrow (\sigma_{\infty}\gamma_{\infty}^{-1})>0.$
\end{assumption}
By virtue of the work of Rockafellar \cite{RRT.1976}, the stopping criterion for step $1$ in Algorithm \ref{Alg:1} is given below:
\begin{align}
	\| (x^{k+1};y^{k+1})-\mathcal{P}_k(x^k,y^k)\|_{{\mathcal{M}}_{k}} & \leq \varsigma_k,\ \varsigma_k\geq 0,\ \sum_{k=0}^{\infty}\varsigma_k<\infty, \label{eqn1}\tag{A}\\
	\| (x^{k+1};y^{k+1}) -\mathcal{P}_k(x^k,y^k) \|_{{\mathcal{M}}_{k}}& \leq \zeta_k\|(x;y)
	-(x^k;Ax^k-\epsilon e)\|_{{\mathcal{M}}_{k}},\ \sum_{k=0}^{\infty}\zeta_k<\infty, \label{eqn2}\tag{B}
\end{align}
where $0\leq\zeta_k\leq 1,\forall k\geq 1.$
Since $\mathcal{P}_k(x^k)$ is difficult to compute in practice, thanks to \cite{ML.DS.KCT.2019}, we utilize the following equivalent stopping criteria in place of stopping criteria \eqref{eqn1} and \eqref{eqn2}:
\begin{align}
	f_k(x^{k+1},y^{k+1})+\psi_k(u^{k+1})&\leq\frac{\varsigma_k^2}{2\sigma_k},\ \varsigma_k\geq 0,\ \sum_{k=0}^{\infty}\varsigma_k<\infty, \label{eqn3}\tag{A'}\\
	f_k(x^{k+1},y^{k+1})+\psi_k(u^{k+1})&\leq \frac{\zeta_k^2}{2\sigma_k}\|(x;y)
	-(x^k;Ax^k-\epsilon e)\|_{{\mathcal{M}}_{k}}^2,\ \sum_{k=0}^{\infty}\zeta_k<\infty,\label{eqn4}\tag{B'}
\end{align}
where  for $k=1,\dots, 0\leq\zeta_k\leq 1,$ and $f_k$ and $\psi_k$ are defined in \eqref{fk} and \eqref{ppadual}, respectively.

The validity of local convergence of the {\sc{PpdSsn}} algorithm depends on the error bound condition, thus we need to verify that problem \eqref{DR3} satisfies the error bound condition. The maximal monotone operator \cite{RTR.1976b} associated with \eqref{DR3} is defined by
$$
\mathcal{T}_f(x,y):=\partial f(x,y)=\begin{bmatrix}
	\frac{1}{2N}e+\frac{1}{2N}\partial\|y\|_1\\
	\partial\chi_{\mathcal{C}}(x)
\end{bmatrix}+\partial\chi_{\widetilde{\mathcal{F}}}(x,y).
$$
According to \cite[Definition 10.20]{RRT.WRJB.1998}, $\delta_{\mathcal{C}}(\cdot)$ is a piecewise linear function. It is obvious that $ \|\cdot\|_1$ is piecewise linear. Thus, by virtue of  \cite[Proposition 12.30]{RRT.WRJB.1998}, $\mathcal{T}_f$ is a piecewise polyhedral multifunction. Then it follows from \cite[Proposition 1 \& Corollary]{RSM.1981} that $\mathcal{T}_f$ satisfies the error bound condition, i.e., there exist $r>0$ and $\kappa>0$ such that if $	{\rm{dist}}((x,y),\mathcal{S}_{P})\leq r$,  then it holds
\begin{equation}\label{error}
	{\rm{dist}}((x,y),\mathcal{S}_{P})\leq \kappa {\rm{dist}}(0,\mathcal{T}_f(x,y)).
\end{equation}

As studied in \cite{XL.DE.KCT.2020}, one  can get the following results on the global and local convergence of the {\sc{PpdSsn}} algorithm.
\begin{theorem}\label{thm4}
	(1) Let $\{(x^k,y^k)\}$ be the sequence generated by the {\sc PpdSsn} algorithm with stopping criterion \eqref{eqn1}. Then $\{(x^k,y^k)\}$ is bounded and 
	\begin{equation}\nonumber
		{\rm{dist}}_{\mathcal{M}_{k+1}}((x^{k+1},y^{k+1}),\mathcal{S}_{P})\leq {\rm{dist}}_{\mathcal{M}_{k}}((x^{k},y^{k}),\mathcal{S}_{P})+\varsigma_k,\ \forall k\geq 0.
	\end{equation}
	In addition, $\{(x^k,y^k)\}$ converges to the optimal point $\{(x^*,y^*)\}$ of problem \eqref{DR3} such that $0\in \mathcal{T}_f(x^*,y^*)$.
	
	(2) Let $r:=\sum_{i=0}^{\infty}\varsigma_k+{\rm{dist}}_{\mathcal{M}_{0}}((x^{0},y^{0}),\mathcal{S}_{P}).$ Then, for this $r>0$, there exists a constant $\kappa>0$ such that $\mathcal{T}_f$ satisfies the error bound condition \eqref{error}. Suppose  $\{(x^{k},y^{k})\}$ is the sequence generated by the {\sc PpdSsn} algorithm with the stopping criteria \eqref{eqn1} and \eqref{eqn2} with nondecreasing $\{\sigma_k\}$. Then it holds that for all $k\geq 0$, 
	\begin{equation*}
		{\rm{dist}}_{\mathcal{M}_{k+1}}((x^{k+1},y^{k+1}),\mathcal{S}_{P})\leq \tilde{\mu}_k{\rm{dist}}_{\mathcal{M}_{k}}((x^{k},y^{k}),\mathcal{S}_{P}),
	\end{equation*}
	where 
	$$
	{\tilde{\mu}}_k=\frac{1}{1-\zeta_k}\left[\zeta_k+\frac{(1+\zeta_k)\kappa\lambda_{\max}(\mathcal{M}_k)}{\sqrt{\sigma_k^2+\kappa^2\lambda_{\max}^2(\mathcal{M}_k)}}\right]
	$$ 
	and
	$$
	\mathop{\lim\sup}\limits_{k\to\infty}{\tilde{\mu}}_k={\tilde{\mu}}_{\infty}=\frac{\kappa\lambda_{\infty}}{\sqrt{\sigma_\infty^2+\kappa^2\lambda_{\infty}^2}}<1
	$$
	with $\lambda_\infty=\mathop{\lim\sup}\limits_{k\to\infty}\lambda_{\max}(\mathcal{M}_k).$ In addition, it holds that for all $k\geq0$, 
	$$
	{\rm{dist}}((x^{k+1},y^{k+1}),\mathcal{S}_{P})\leq \frac{\tilde{\mu}_k}{\sqrt{\lambda_{\min}(\mathcal{M}_{k+1})}}{\rm{dist}}_{\mathcal{M}_{k}}((x^{k},y^{k}),\mathcal{S}_{P}).
	$$
\end{theorem}
\begin{proof}
	The first part of the proof follows from the work of Li et al.\cite[proposition 2.3]{XL.DE.KCT.2020}. Due to \cite[proposition 2.5]{XL.DE.KCT.2020}, we get  the desired result.
\end{proof}
\subsection{A semismooth Newton method for the subproblem}
In this subsection,  we shall describe the details of the semismooth Newton method \cite{BK.1988,LQ.JS.1993,DS.JS.2002} for solving subproblem \eqref{ppadual}, which is regarded as the most computationally expensive part of the Algorithm \ref{Alg:1}.  For any given $\sigma,\gamma>0$ and $\tilde{x}\in\mathbb{R}^m,$ our purpose is to solve the problem defined in \eqref{ppadual}. It is clear that $\psi(\cdot)$ is convex and continuously differentiable, whose gradient is given by
\begin{equation*}
	\nabla\psi(u)=-\epsilon e+A\Pi_{\mathcal{C}}(\tilde{x}+\sigma A^{\top}u)-{\rm{Prox}}_{\frac{\gamma_k}{2N}\|\cdot\|_1}(A\tilde{x}-\epsilon e-\gamma(\frac{1}{2N}e+u)).
\end{equation*}
Thus, problem \eqref{ppadual} admits an optimal solution $u^*$, which can be obtained by solving the following nonsmooth equations:
$$
\nabla\psi(u)=0.
$$
In order to solve the above equations, we need to characterize the generalized Jacobian of $\nabla\psi(\cdot)$, which depends on the generalized Jacobian of ${\rm{Prox}}_{\frac{\gamma_k}{2N}\|\cdot\|_1}(\cdot)$ and $\Pi_{\mathcal{C}}(\cdot)$. Since the proximal mapping ${\rm{Prox}}_{\frac{\gamma_k}{2N}\|\cdot\|_1}(\cdot)$ and the projection $\Pi_{\mathcal{C}}(\cdot)$ are Lipschitz continuous, we define the multifunction $\hat{\partial}^2\psi(\cdot)$ by
$$
\hat{\partial}^2\psi(u):=\sigma A\partial\Pi_{\mathcal{C}}(\tilde{x}+\sigma A^{\top}u)A^{\top}+\gamma \partial {\rm{Prox}}_{\frac{\gamma_k}{2N}\|\cdot\|_1}(A\tilde{x}-\epsilon e-\gamma(\frac{1}{2N}e+u)),
$$
where $\partial{\rm{Prox}}_{\frac{\gamma_k}{2N}\|\cdot\|_1}(A\tilde{x}-\epsilon e-\gamma(\frac{1}{2N}e+u))$ denotes the Clarke subdifferential \cite{FHC.1983} of the proximal mapping ${\rm{Prox}}_{\frac{\gamma_k}{2N}\|\cdot\|_1}(\cdot)$ at $A\tilde{x}-\epsilon e-\gamma(\frac{1}{2N}e+u).$
Let $U\in\partial{\rm{Prox}}_{\frac{\gamma_k}{2N}\|\cdot\|_1}(\tilde{z})$ with $\tilde{z}:=A\tilde{x}-\epsilon e-\gamma(\frac{1}{2N}e+u)$. Then we can obtain that $U = {\rm{Diag}}(\bar{u}_1, ..., \bar{u}_N)$ with
$$
\bar{u}_i = \left\{
\begin{aligned}
	&1,\ |\tilde{z}|_i>\frac{\gamma}{2N}, \\
	&0,\ |\tilde{z}|_i\leq\frac{\gamma}{2N}.
\end{aligned}
\right.
$$
On the other hand, it is difficult to characterize the B-subdifferential $\partial_B\Pi_{\mathcal{C}}(\tilde{x}+\sigma A^{\top}u)$ or the Clarke generalized Jacobian $\partial\Pi_{\mathcal{C}}(\tilde{x}+\sigma A^{\top}u)$ of the projection $\Pi_{\mathcal{C}}(\tilde{x}+\sigma A^{\top}u)$ at $\tilde{x}+\sigma A^{\top}u$. Thus, we utilize the  generalized HS-Jacobian proposed by Han and Sun  \cite{HJY.SDF.1997}  to replace the calculation of  $\partial\Pi_{\mathcal{C}}(\tilde{x}+\sigma A^{\top}u)$. 
\subsubsection{The generalized HS-Jacobian of the projection $\Pi_{\mathcal{C}}(\cdot)$}
In this subsection, we shall characterize the generalized HS-Jacobian  of the projection $\Pi_{\mathcal{C}}(\cdot)$ at $\tilde{x}+\sigma A^{\top}u$.

Denote 
$
\hat{z}:=\tilde{x}+\sigma A^{\top}u\in\mathbb{R}^m. 
$
We define the following index subsets of $\{1,\dots,m\}$:
$$
\mathcal{K}_1:=\{i\ |\ (\Pi_{\mathcal{C}}(\hat{z}))_i=0\},\ \mathcal{K}_2:=\{1,\dots,m\}\backslash\mathcal{K}_1.
$$
From the definition of $\mathcal{C}$  we know that  $|\mathcal{K}_2|\neq0$. We also have $ |\mathcal{K}_1|+|\mathcal{K}_2|=m.$  By virtue of the results obtained in \cite{YJL.WMZ.2023}, we give below the explicit formulas for the generalized HS-Jacobian of $\Pi_{\mathcal{C}}(\cdot)$.
\begin{theorem}\label{thm5}
	Assume that $\hat{\mu}\in\mathbb{R}^m,\rho,\epsilon\in\mathbb{R}$ in problem \eqref{DR3} are given. For given $\hat{z}\in\mathbb{R}^m,$ denote 
	$$
	{w}_i=\left\{
	\begin{aligned}
		&1,\ i\in\mathcal{K}_2, \\
		&0,\ {\rm{otherwise}},
	\end{aligned}
	\right.
	({e}_{\mathcal{K}_2}^m)_i=\left\{
	\begin{aligned}
		&1,\ i\in\mathcal{K}_2, \\
		&0,\ {\rm{otherwise}},
	\end{aligned}
	\right.
	(\hat{\mu}_{\mathcal{K}_2}^m)_i=\left\{
	\begin{aligned}
		&\hat{\mu}_i,\ i\in\mathcal{K}_2, \\
		&0,\ {\rm{otherwise}},
	\end{aligned}
	\right.  i=1,2,\dots,m.
	$$
	Then, the element ${N}_0$ of the generalized HS-Jacobian
	for $\Pi_{\mathcal{C}}(\cdot)$ at $\hat{z}$ admits the following explicit expressions:
	
	\noindent{\bf{\uppercase\expandafter{\romannumeral1}}}. If $\hat{\mu}^{\top}\Pi_{\mathcal{C}}(\hat{z})\neq \rho+\epsilon$, then
	$$
	{N}_0 = {\rm{Diag}}({w})-\frac{1}{|\mathcal{K}_2|}{e}_{\mathcal{K}_2}^m({e}_{\mathcal{K}_2}^m)^{\top}.
	$$
	
	\noindent{\bf{\uppercase\expandafter{\romannumeral2}}}. If $\hat{\mu}^{\top}\Pi_{\mathcal{C}}(\hat{z})= \rho+\epsilon$, then the following two cases are taken into consideration. Denote 
	$$
	\eta:=\|\hat{\mu}_{\mathcal{K}_2}\|^2|\mathcal{K}_2|-(\hat{\mu}_{\mathcal{K}_2}^{\top}{e}_{\mathcal{K}_2})^2.
	$$
	
	\noindent{\bf{(\romannumeral1)}} If $\eta\neq 0,$ then
	$$
	\begin{aligned}
		&{N}_0={\rm{Diag}}({w})\\
		&\scriptsize-\frac{1}{\eta}\begin{bmatrix}
			-\sqrt{|\mathcal{K}_2|}(\hat{\mu}_{\mathcal{K}_2}^m)^{\top}+\frac{\hat{\mu}_{\mathcal{K}_2}^{\top}{e}_{\mathcal{K}_2}}{\sqrt{|\mathcal{K}_2|}}(e_{\mathcal{K}_2}^m)^{\top}\\
		\sqrt{\|\hat{\mu}_{\mathcal{K}_2}\|^2-\frac{(\hat{\mu}_{\mathcal{K}_2}^{\top}{e}_{\mathcal{K}_2}^2)}{|{\mathcal{K}_2}|}}(e_{\mathcal{K}_2}^m)^{\top}
		\end{bmatrix}^{\top}\begin{bmatrix}
			-\sqrt{|\mathcal{K}_2|}(\hat{\mu}_{\mathcal{K}_2}^m)^{\top}+\frac{\hat{\mu}_{\mathcal{K}_2}^{\top}{e}_{\mathcal{K}_2}}{\sqrt{|\mathcal{K}_2|}}(e_{\mathcal{K}_2}^m)^{\top}\\
			\sqrt{\|\hat{\mu}_{\mathcal{K}_2}\|^2-\frac{(\hat{\mu}_{\mathcal{K}_2}^{\top}{e}_{\mathcal{K}_2}^2)}{|{\mathcal{K}_2}|}}(e_{\mathcal{K}_2}^m)^{\top}
		\end{bmatrix}.
	\end{aligned}
	$$
	
	\noindent{\bf{(\romannumeral2)}} If $\eta=0$, i.e., $\hat{\mu}_{\mathcal{K}_2}=e_{\mathcal{K}_2}$, then
	$$
{N}_0 = {\rm{Diag}}({w})-\frac{1}{|\mathcal{K}_2|}{e}_{\mathcal{K}_2}^m({e}_{\mathcal{K}_2}^m)^{\top}.
$$
\end{theorem}


\subsubsection{Constraint nondegeneracy condition for problem \eqref{ppadual}}
Let $\left( \widehat{x},u^*\right) \in\mathbb{R}^m\times\mathbb{R}^N$ be an optimal solution pair  of problem \eqref{DR3} and \eqref{ppadual}.  In order to establish local convergence of the semismooth Newton algorithm  for problem \eqref{ppadual}, we need to study the positive definiteness of $N_{\widehat{x}}\in\mathcal{N}_{\mathcal{C}}(\widehat{x})$. Naturally, we have 
$$
\widehat{x}=\Pi_{\mathcal{C}}(x^k+\sigma_kA^{\top}u^*).
$$
For later analysis, we intend to characterize the tangent cone $\mathscr{T}_{\mathcal{C}}(\widehat{x})$ of the set $\mathcal{C}$. 
Denote $\widehat{\mathcal{K}}_1:=\{i\ |\ (\Pi_{\mathcal{C}}(\widehat{x}))_i=0\},\ \widehat{\mathcal{K}}_2:=\{1,\dots,m\}\backslash\widehat{\mathcal{K}}_1$. For the tangent cone $\mathscr{T}_{\mathcal{C}}(\widehat{x})$,  we consider the following two cases:
\begin{itemize}
\item [(1)] If $\hat{\mu}^{\top}\Pi_{\mathcal{C}}(\widehat{x})\neq \rho+\epsilon$, then 
$$
\mathscr{T}_{\mathcal{C}}(\widehat{x}) = \left\lbrace d\in\mathbb{R}^m\mid e^{\top}d = 0,\ d_{\widehat{\mathcal{K}}_1}\geq0\right\rbrace.
$$
The linearity space  of $\mathscr{T}_{\mathcal{C}}(\widehat{x})$ is given by
$$
{\rm{lin}}\left( \mathscr{T}_{\mathcal{C}}(\widehat{x})\right) = \left\lbrace d\in\mathbb{R}^m\mid e^{\top}d = 0, d_{\widehat{\mathcal{K}}_1} = 0\right\rbrace = \left\lbrace d\in\mathbb{R}^m\mid  e_{\widehat{\mathcal{K}}_2}^{\top}d_{\widehat{\mathcal{K}}_2} = 0,d_{\widehat{\mathcal{K}}_1} = 0\right\rbrace.
$$
Moreover, we easily get the orthogonal space of $ {\rm{lin}}\left( \mathscr{T}_{\mathcal{C}}(\widehat{x})\right) $, which is characterized by
$$
\left( {\rm{lin}}\left( \mathscr{T}_{\mathcal{C}}(\widehat{x})\right) \right) ^{\bot}=\left\lbrace h\in\mathbb{R}^m\mid h_{\widehat{\mathcal{K}_2}} = ke_{\widehat{\mathcal{K}_2}},k\in \mathbb{R}\right\rbrace.
$$
\item [(2)] If $\hat{\mu}^{\top}\Pi_{\mathcal{C}}(\widehat{x})=\rho+\epsilon$, then 
$$
\mathscr{T}_{\mathcal{C}}(\widehat{x}) = \left\lbrace d\in\mathbb{R}^m\mid e^{\top}d = 0,\ d_{\widehat{\mathcal{K}}_1}\geq0,\hat{\mu}^{\top}d \geq0\right\rbrace.
$$
The linearity space  of $\mathscr{T}_{\mathcal{C}}(\widehat{x})$ has the following form:
$$
{\rm{lin}}\left( \mathscr{T}_{\mathcal{C}}(\widehat{x})\right)  = \left\lbrace d\in\mathbb{R}^m\mid  e_{\widehat{\mathcal{K}}_2}^{\top}d_{\widehat{\mathcal{K}}_2} = 0,\hat{\mu}_{\widehat{\mathcal{K}}_2}^{\top}d_{\widehat{\mathcal{K}}_2} =0,d_{\widehat{\mathcal{K}}_1} = 0\right\rbrace,
$$
and its  orthogonal space is given by 
$$
\left( {\rm{lin}}\left( \mathscr{T}_{\mathcal{C}}(\widehat{x})\right) \right) ^{\bot}=\left\lbrace h\in\mathbb{R}^m\mid h_{\widehat{\mathcal{K}_2}} = k_1e_{\widehat{\mathcal{K}_2}}+k_2\hat{\mu}_{\widehat{\mathcal{K}_2}}, \ k_1,k_2\in \mathbb{R}\right\rbrace.
$$
\end{itemize}
The constraint nondegeneracy condition \cite{BJF.SA} holds at $\widehat{x}$ for problem \eqref{DR3} if
$$
\left[ \begin{matrix}
A\\
I_m
\end{matrix}\right]  \mathbb{R}^m+\left[ \begin{matrix}
\left\lbrace 0\right\rbrace^N\\
{\rm{lin}}\left(  \mathscr{T}_{\mathcal{C}}(\widehat{x})\right) 
\end{matrix}\right] =\left[ \begin{matrix}
\mathbb{R}^N\\
\mathbb{R}^m
\end{matrix}\right] ,
$$
or equivalently,
\begin{equation}\label{CQ}
A{\rm{lin}}\left( \mathscr{T}_{\mathcal{C}}(\widehat{x})\right) = \mathbb{R}^N.
\end{equation}

Before establishing the connection between the constraint nondegeneracy condition and the positive definiteness of $\mathcal{N}_{\mathcal{C}}(\widehat{x})$, we need to  state the following crucial  lemma.
\begin{lemma}\label{lemma1}
For any $N_{\widehat{x}}\in\mathcal{N}_{\mathcal{C}}(\widehat{x})$ and $h\in\mathbb{R}^m$ such that $N_{\widehat{x}}h=0$, it holds that
$$
h\in \left( {\rm{lin}}\left( \mathscr{T}_{\mathcal{C}}(\widehat{x})\right) \right) ^{\bot}.
$$
\end{lemma}
\begin{proof}
Let $N_{\widehat{x}}\in\mathcal{N}_{\mathcal{C}}(\widehat{x})$ and $h\in\mathbb{R}^m$ satisfy $N_{\widehat{x}}h=0$. Recalling  Theorem \ref{thm5}, we take the following cases into account:
\begin{itemize}
	\item [(1)]  If $\hat{\mu}^{\top}\Pi_{\mathcal{C}}(\widehat{x})\neq \rho+\epsilon$, then 
	$$
	0=\left\langle h,N_{\widehat{x}}h\right\rangle =\|h_{\mathcal{K}_2}\|^2-\frac{1}{|\mathcal{K}_2|}(h^{\top}_{\mathcal{K}_2}e_{\mathcal{K}_2})^2.
	$$
	It is clear that $\|h_{\mathcal{K}_2}\|^2-\frac{1}{|\mathcal{K}_2|}(h^{\top}_{\mathcal{K}_2}e_{\mathcal{K}_2})^2=0$ holds if and only if $h_{\mathcal{K}_2}=ke_{\mathcal{K}_2},k\in\mathbb{R},$ which implies $h\in \left( {\rm{lin}}\left( \mathscr{T}_{\mathcal{C}}(\widehat{x})\right) \right) ^{\bot}.
	$
	\item[(2)]  If $\hat{\mu}^{\top}\Pi_{\mathcal{C}}(\hat{y})= \rho+\epsilon$, then we consider cases {\bf{(\romannumeral1)}} and {\bf{(\romannumeral2)}} for Case {\bf{\uppercase\expandafter{\romannumeral1}}} of Theorem \ref{thm5}. For case {\bf{(\romannumeral1)}}, we denote $$
	\tilde{A}=\begin{bmatrix}
		-\sqrt{|\mathcal{K}_2|}\hat{\mu}_{\mathcal{K}_2}^m+\frac{\hat{\mu}_{\mathcal{K}_2}^{\top}{e}_{\mathcal{K}_2}}{\sqrt{|\mathcal{K}_2|}}e_{\mathcal{K}_2}^m& \sqrt{\|\hat{\mu}_{\mathcal{K}_2}\|^2-\frac{(\hat{\mu}_{\mathcal{K}_2}^{\top}{e}_{\mathcal{K}_2}^2)}{|{\mathcal{K}_2}|}}e_{\mathcal{K}_2}^m
	\end{bmatrix}.
	$$
	Since $N_{\widehat{x}}$ has an eigenvalue of $1$ or $0$, we know that $\tilde{A}\tilde{A}^{\top}$ has an eigenvalue of $\eta$ or $0$
	Then we have 
	$$
	\begin{aligned}
		\left\langle h,N_{\widehat{x}}h\right\rangle&=\|h_{\mathcal{K}_2}\|^2
		-\frac{1}{\eta}h^{\top}_{\mathcal{K}_2}\tilde{A}\tilde{A}^{\top}h_{\mathcal{K}_2}\\
		&\geq \|h_{\mathcal{K}_2}\|^2-\frac{1}{\eta}\|\tilde{A}\|^2_2\|h_{\mathcal{K}_2}\|^2=\|h_{\mathcal{K}_2}\|^2-\frac{1}{\eta}\widetilde{\lambda}_{max}(\tilde{A}\tilde{A}^{\top})\|h_{\mathcal{K}_2}\|^2=0.
	\end{aligned}
	$$
	Thus, $\left\langle h,N_{\widehat{x}}h\right\rangle=0$ if and only if 
	$$ 
	d_{\mathcal{K}_2}=\frac{1}{\sqrt{2\eta}}\left[ -\sqrt{|\mathcal{K}_2|}\hat{\mu}_{\mathcal{K}_2}+\frac{\hat{\mu}_{\mathcal{K}_2}^{\top}{e}_{\mathcal{K}_2}}{\sqrt{|\mathcal{K}_2|}}e_{\mathcal{K}_2}\right]+\frac{1}{\sqrt{2\eta}}\left[ \sqrt{\|\hat{\mu}_{\mathcal{K}_2}\|^2-\frac{(\hat{\mu}_{\mathcal{K}_2}^{\top}{e}_{\mathcal{K}_2}^2)}{|{\mathcal{K}_2}|}}e_{\mathcal{K}_2}\right],
	$$
	which means that
	$h\in \left( {\rm{lin}}\left( \mathscr{T}_{\mathcal{C}}(\widehat{x})\right) \right) ^{\bot}.
	$
	For case {\bf{(\romannumeral1)}},   we  know from (1) that $ h\in \left( {\rm{lin}}\left( \mathscr{T}_{\mathcal{C}}(\widehat{x})\right) \right) ^{\bot}.$
	Here, we complete the proof.
\end{itemize}
\end{proof}

\begin{proposition}\label{CQSD}
Let $\left( \widehat{x},u^*\right) \in\mathbb{R}^m\times\mathbb{R}^N$ be an optimal solution pair  of problem \eqref{DR3} and \eqref{ppadual} with $\widehat{x}=\Pi_{\mathcal{C}}(x^k+\sigma_kA^{\top}u^*)$.  Then the following conditions are equivalent:
\begin{itemize}
	\item [(a)] The constraint nondegeneracy condition \eqref{CQ} holds at $\widehat{x}$;
	
	\item [(b)] $AN_{\widehat{x}}A^{\top}\in A\mathcal{N}_{\mathcal{C}}(\widehat{x})A^{\top}$ is symmetric and positive definite;
	
	\item [(c)]$AN_0A^{\top}\in A\mathcal{N}_{\mathcal{C}}(\widehat{x})A^{\top}$ is symmetric and positive definite.
\end{itemize}
\end{proposition}
\begin{proof}
$``(a)\Rightarrow (b)"$. Let $N_{\widehat{x}}$ be an arbitrary element in $\mathcal{N}_{\mathcal{C}}(\widehat{x})$. Let $h\in\mathbb{R}^m$ such that $AN_{\widehat{x}}A^{\top}h=0$. Obviously, $AN_{\widehat{x}}A^{\top}$ is symmetric. Due to\cite[Proposition 2.1]{JKF.SDF.TKC.2012}, we have 
$$
0=\langle  h, AN_{\widehat{x}}A^{\top}h\rangle  =\langle A^{\top}h, N_{\widehat{x}}A^{\top}h\rangle \geq \langle N_{\widehat{x}}A^{\top}h, N_{\widehat{x}}A^{\top}h\rangle ,
$$
which implies $N_{\widehat{x}}A^{\top}h=0$. By virtue of Lemma \ref{lemma1},  we have $A^{\top}h\in\left( {\rm{lin}}\left( \mathscr{T}_{\mathcal{C}}(\widehat{x})\right) \right) ^{\bot}.$ Since the constraint nondegeneracy condition \eqref{CQ} holds at $\widehat{x}$, there exists $x\in {\rm{lin}}\left( \mathscr{T}_{\mathcal{C}}(\widehat{x})\right) $ such that $Ax = h$. Then we obtain 
$$
\left\langle h,h\right\rangle =\langle  h, Ax \rangle =\langle  A^{\top}h,x
\rangle =0.
$$
Thus, we have $h=0$. Therefore, $AN_{\widehat{x}}A^{\top}$ is positive definite.

$``(b)\Rightarrow (c)".$ This is obviously true since $AN_0A^{\top}\in A\mathcal{N}_{\mathcal{C}}(\widehat{x})A^{\top}.$ 

$``(c)\Rightarrow (a)".$ Assume that the constraint nondegeneracy condition \eqref{CQ} does not hold at $\widehat{x}$, i.e., 
$$
\left( A{\rm{lin}}\left( \mathscr{T}_{\mathcal{C}}(\widehat{x})\right) \right) ^{\bot}\neq \left\lbrace 0\right\rbrace^N.
$$
Let $0\neq h\in\left( A{\rm{lin}}\left( \mathscr{T}_{\mathcal{C}}(\widehat{x})\right) \right) ^{\bot}$. Then 
$$
\left\langle h,Ax\right\rangle=\langle A^{\top}h,x\rangle  =0,\ \forall x\in{\rm{lin}}\left( \mathscr{T}_{\mathcal{C}}(\widehat{x})\right).
$$
Denote $d := A^{\top}h$. Then we have $d\in\left( {\rm{lin}}\left( \mathscr{T}_{\mathcal{C}}(\widehat{x})\right)\right) ^{\bot}.$
If $\hat{\mu}^{\top}\Pi_{\mathcal{C}}(\widehat{x})\neq \rho+\epsilon$, then $d_{\mathcal{K}_2}=ke_{\mathcal{K}_2},k\in\mathbb{R}.$ Thus we have $N_{\widehat{x}}d=0.$
If $\hat{\mu}^{\top}\Pi_{\mathcal{C}}(\widehat{x})=\rho+\epsilon$ and $	\eta=\|\hat{\mu}_{\mathcal{K}_2}\|^2|\mathcal{K}_2|-(\hat{\mu}_{\mathcal{K}_2}^{\top}{e}_{\mathcal{K}_2})^2\neq0$ , then without loss of generality,
we choose 
$$
d_{\mathcal{K}_2}=\frac{1}{\sqrt{2\eta}}\left[ -\sqrt{|\mathcal{K}_2|}\hat{\mu}_{\mathcal{K}_2}+\frac{\hat{\mu}_{\mathcal{K}_2}^{\top}{e}_{\mathcal{K}_2}}{\sqrt{|\mathcal{K}_2|}}e_{\mathcal{K}_2}\right]+\frac{1}{\sqrt{2\eta}}\left[ \sqrt{\|\hat{\mu}_{\mathcal{K}_2}\|^2-\frac{(\hat{\mu}_{\mathcal{K}_2}^{\top}{e}_{\mathcal{K}_2}^2)}{|{\mathcal{K}_2}|}}e_{\mathcal{K}_2}\right]
$$
in $\left( {\rm{lin}}\left( \mathscr{T}_{\mathcal{C}}(\widehat{x})\right)\right) ^{\bot}$.
Similarly, we get $N_{\widehat{x}}d=0.$
Thus, 
$$
\left\langle d,N_{\widehat{x}}d\right\rangle =\langle A^{\top}h,N_{\widehat{x}} A^{\top}h \rangle =\langle h,AN_{\widehat{x}}A^{\top}h\rangle=0.
$$
Since $AN_{\widehat{x}}A^{\top}$ is positive definite, it follows that $h=0$, which  contradicts to the assumption that $h\neq0$. Here, we know that (a) holds.
\end{proof}

\subsubsection{A semismooth Newton algorithm and its convergence}
Since ${\rm{Prox}}_{\frac{\gamma_k}{2N}\|\cdot\|_1}(\cdot)$ and $\Pi_{\mathcal{C}}(\cdot)$ are strongly semismooth, $\nabla\psi(\cdot)$ is strongly semismooth. Thus, the semismooth Newton algorithm can be adapted to problem \eqref{ppadual}, which is  depicted in Algorithm \ref{Alg:2}.
\begin{algorithm}[H]
\caption{ {({\sc Ssn}) A semismooth Newton algorithm for solving problem \eqref{ppadual} }}\label{Alg:2}	
\begin{algorithmic}[1]
	\Require
	$\vartheta\in(0,1/2),\bar \tau\in (0,1]$, $\tilde{\upsilon}_1\in\left( 1,+\infty\right) $ and $\bar\varrho,\varrho,{\tilde{\upsilon}}_2\in(0,1)$. Given an initial point ${u^0\in\mathbb{R}^N}$ and set  $j=0$.
	\Ensure The approximate solution $u^{j+1}$ of problem \eqref{ppadual}.
	\State Choose $N_j\in\mathcal{N}_{\mathcal{C}}(\hat{z}^j)$, $U_j\in\partial {\rm{Prox}}_{\frac{\gamma_k}{2N}\|\cdot\|_1}(\tilde{z}^j)$ and $\epsilon_j:=\tilde{\upsilon}_1\min\left\lbrace \tilde{\upsilon}_2,\|\nabla\psi_k(u^j)\|\right\rbrace$. Denote $\mathcal{V}_j:=\sigma AN_jA^{\top}+\gamma U_j$. Solve the following linear system 
	\begin{equation}\label{newton}
		\begin{split}
			(\mathcal{V}_j+\epsilon_jI_N)d= -\nabla\psi_k(u^{j})
		\end{split}
	\end{equation}
	by the direct method or the conjugate gradient method such that the approximate solution $d^j\in\mathbb{R}^m$ satisfies
	$$
	\|(\mathcal{V}_j+\epsilon_j I_N)d^j+\nabla\psi_k(u^{j})\|\leq \min(\bar \varrho,\|\nabla\psi_k(u^{j})\|^{1+\bar\tau}).
	$$
	\State Set $\alpha_j=\varrho^{\bar{c}_j}$, where $\bar{c}_j$ is the smallest nonnegative integer $\bar{c}$ such that  
	$$
	\psi_k(u^{j}+\varrho^{\bar{c}}d^j)\leq\psi_k(u^{j})+\vartheta\varrho^{\bar{c}}\langle\nabla\psi_k(u^{j}),d^j\rangle.
	$$
	\State Update $u^{j+1}=u^{j}+\alpha_jd^j,\ j\gets j+1$,  go to step $1$.
\end{algorithmic}
\end{algorithm}
Since $\mathcal{V}_j$ is positive semidefinite,  $\mathcal{V}_j+\epsilon_jI_N$ is always positive definite as long as $\|\nabla\psi_k(u^j)\|\neq 0$. 
As stated in  \cite[Theorem 3.4 and 3.5]{XYZ.DFS.KCT.2010}, the convergence result of Algorithm \ref{Alg:2} can be described below.

\begin{theorem}
Let $u^*$ be an accumulation point of the infinite sequence $\{u^j\}$  generated by {\sc Ssn} algorithm. Assume that the constraint nondegeneracy condition \eqref{CQ} holds at $\widehat{x}:=\Pi_{\mathcal{C}}(x^k+\sigma_kA^{\top}u^*)$.  Then the sequence $\{u^j\}$ converges to the unique optimal solution ${u}^*$ of problem \eqref{ppadual}. Moreover, the rate  of convergence is at least superlinear with
$$
\|u^{j+1}-{u}^*\|=\mathcal{O}(\|u^j-{u}^*\|^{1+\bar{\tau}}),
$$
where $\bar{\tau}$ is given in the {\sc Ssn} algorithm.
\end{theorem}


\subsubsection{Efficient implementations of the linear system \eqref{newton}}
As we can observe, the most computationally expensive step in semismooth Newton algorithm  is to solve linear systems.
In the following we shall show an efficient implementation of the semismooth Newton method, which is achieved by fully exploiting the sparsity of the generalized Jacobian to efficiently reduce the computational cost.

Given $\sigma,\gamma,\epsilon_j>0$ and $(u,\tilde{x})\in\mathbb{R}^N\times\mathbb{R}^m.$ Choosing $N\in\mathcal{N}_{\mathcal{C}}(\hat{z})$ with $\hat{z}=\tilde{x}+\sigma A^{\top}u$ and $U\in\partial {\rm{Prox}}_{\frac{\gamma}{2N}\|\cdot\|_1}(\tilde{z})$ with $\tilde{z}=A\tilde{x}-\epsilon e-\gamma(\frac{1}{2N}e+u)$, we consider the following  linear system:
\begin{equation}\label{linear}
(\epsilon_jI_N+\gamma U+\sigma ANA^{\top})d=-\nabla\psi(u).
\end{equation}

The cost of naively computing $ANA^{\top}$ and $ANA^{\top}d$ are $O(m^2N+N^2m)$ and $O(Nm+m^2)$, respectively. Therefore, when $m$ and $N$ are large, it is expensive to compute the matrix $\epsilon_jI_N+\gamma U+\sigma ANA^{\top}$ directly. For large scale linear systems, Cholesky decomposition and conjugate gradient methods are not desirable. In our implementation, we choose $N=N_0$ obtained in Theorem \ref{thm5} and reduce the cost by exploiting the special structure of $N_0$.

Thanks to Theorem \ref{thm5}, when $\hat{\mu}^{\top}\Pi_{\mathcal{C}}(\hat{z})\neq \rho+\epsilon$, we have 
$$
{A}N{A}^{\top}={A}\left[{\rm{Diag}}({w})-\frac{1}{|\mathcal{K}_2|}{e}_{\mathcal{K}_2}^n({e}_{\mathcal{K}_2}^n)^{\top}\right] {A}^{\top}={A}_{\mathcal{K}_2}{A}_{\mathcal{K}_2}^{\top}-\frac{1}{|\mathcal{K}_2|}({A}_{\mathcal{K}_2}e_{\mathcal{K}_2})({A}_{\mathcal{K}_2}e_{\mathcal{K}_2})^{\top},
$$
where ${A}_{\mathcal{K}_2}$ is the submatrix of ${A}$ by extracting those columns with indices in $\mathcal{K}_2$. 
Recalling Case {\bf{\uppercase\expandafter{\romannumeral2}}} from Theorem \ref{thm5}. When $\eta=\|\hat{\mu}_{\mathcal{K}_2}\|^2|\mathcal{K}_2|-(\hat{\mu}_{\mathcal{K}_2}^{\top}{e}_{\mathcal{K}_2})^2\neq 0$, we know
$$
\begin{aligned}
&{A}N{A}^{\top}=
{A}_{\mathcal{K}_2}{A}_{\mathcal{K}_2}^{\top}\\
&{\small-\frac{A_{\mathcal{K}_2}}{\eta}\begin{bmatrix}
		-\sqrt{|\mathcal{K}_2|}(\hat{\mu}_{\mathcal{K}_2})^{\top}+\frac{\hat{\mu}_{\mathcal{K}_2}^{\top}{e}_{\mathcal{K}_2}}{\sqrt{|\mathcal{K}_2|}}(e_{\mathcal{K}_2})^{\top}\\
		\sqrt{\|\hat{\mu}_{\mathcal{K}_2}\|^2-\frac{(\hat{\mu}_{\mathcal{K}_2}^{\top}{e}_{\mathcal{K}_2}^2)}{|{\mathcal{K}_2}|}}(e_{\mathcal{K}_2})^{\top}
	\end{bmatrix}^{\top}\begin{bmatrix}
		-\sqrt{|\mathcal{K}_2|}(\hat{\mu}_{\mathcal{K}_2})^{\top}+\frac{\hat{\mu}_{\mathcal{K}_2}^{\top}{e}_{\mathcal{K}_2}}{\sqrt{|\mathcal{K}_2|}}(e_{\mathcal{K}_2})^{\top}\\
		\sqrt{\|\hat{\mu}_{\mathcal{K}_2}\|^2-\frac{(\hat{\mu}_{\mathcal{K}_2}^{\top}{e}_{\mathcal{K}_2}^2)}{|{\mathcal{K}_2}|}}(e_{\mathcal{K}_2})^{\top}
	\end{bmatrix}A^{\top}_{\mathcal{K}_2}.}
\end{aligned}
$$
On the other hand, when $\eta=\|\hat{\mu}_{\mathcal{K}_2}\|^2|\mathcal{K}_2|-(\hat{\mu}_{\mathcal{K}_2}^{\top}{e}_{\mathcal{K}_2})^2= 0$, we have 
$$
{A}N{A}^{\top}={A}_{\mathcal{K}_2}{A}_{\mathcal{K}_2}^{\top}-\frac{1}{|\mathcal{K}_2|}({A}_{\mathcal{K}_2}e_{\mathcal{K}_2})({A}_{\mathcal{K}_2}e_{\mathcal{K}_2})^{\top}.
$$

Based on the above calculation,  the cost of computing ${A}N{A}^{\top}$ and ${A}N{A}^{\top}d$ are  reduced to $O(N^2(|\mathcal{K}_2|+1)+N|\mathcal{K}_2|)$ and $N(|\mathcal{K}_2|+1)$, respectively. When $N$ is small or moderate, we  prefer to directly use  the  Cholesky factorization to solve the linear system, where the total computational cost is $O(N^2(|\mathcal{K}_2|+1)+N|\mathcal{K}_2|)+O(N^3).$
\section{Numerical experiments}\label{sec:5}
In this section, we conduct several experiments on real market data sets to show the out-of-sample performance of the DR-MLSAD model, we also compare the efficiency of different algorithms in solving the DR-MLSAD model on real data sets and random data sets. All our experiments are executed in MATLAB R2019a  on a Dell desktop  computer with Intel Xeon  Gold 6144 CPU @ 3.50GHz  and 256 GB RAM.

\subsection{Comparison of out-of-sample performance}
This subsection  is devoted to comparing the out-of-sample performance of DR-MLSAD models with Wasserstein radius chosen by different methods.  In Section \ref{5.1.1}, we describe the classical model compared to DR-MLSAD model, the collection of real data sets and the setup of procedures for testing out-of-sample performance.  In Section \ref{5.1.2}, we present the performance  of DR-MLSAD models with radius chosen by  RWPI approach and  cross-validation, respectively, while comparing them in performance with some classical models.
\subsubsection{Methodology}\label{5.1.1}
\begin{itemize}
\item[(1)] {\emph{Comparison with Classical Models.}}\ 
We contrast the DR-MLSAD model with the radius selected by RWPI approach (denoted as RWPI-DRMLSAD) and DR-MLSAD model with the radius selected by the cross-validation (denoted as  CV-DRMLSAD). In addition, we compare the out-of-sample performance of the Wasserstein DR-MLSAD model with the SAA model \eqref{saa} and the Naive $1/N$ strategy \cite{DV.GL.UR.2009}. 
For the Naive $1/N$ strategy, it is a simple portfolio model where each asset has an equal weight, i.e., $x_i=1/m,\ i=1,\dots,m.$
\item[(2)] {\emph{Data Sets.}}\ 
We collect the following  six data sets of monthly stock returns from July 1963 to May 2023 from the Ken French's website\footnote{\url{https://mba.tuck.dartmouth.edu/pages/faculty/ken.french/data_library.html}} to validate the numerical results of the proposed model. The six data sets include:  (1) $100$ portfolios formed on momentum ({\emph{10MOM}}); (2) $25$ portfolios formed on size and operating profitability ({\emph{25MEOP}}); (3) $6$ portfolios formed on size and operating profitability ({\emph{6MEOP}}); (4) $100$  portfolios formed on size and operating profitability  ({\emph{100MEOP}}); (5) $25$  portfolios formed on book-to-market and operating profitability ({\emph{25BEMEOP}}); (6) $100$ portfolios formed on size and investment ({\emph{100MEINV}}). In all cases, we remove those assets that have missing values. The statistics of all test data sets are presented in Table \ref{Tab:1}.
	\begin{table}[htb]
	\centering
	\caption{Summary of tested data sets} \label{Tab:1}
	\setlength{\tabcolsep}{2mm}{
		\begin{tabular}{ccccccc}
			\hline	 \multirow{1}*{No. }  &\multirow{1}*{Dataset}&\multirow{1}*{Stocks}& \multicolumn{1}{c}{Samples} & \multicolumn{1}{c}{Time period } & \multicolumn{1}{c}{Source} & \multicolumn{1}{c}{Frequency } \\
			\midrule 
			1 & {\emph{10MOM}} & 10 &719& 07/1963-05/2023  &   K.French & Monthly \\
			2 & {\emph{25MEOP}} & 25 &719& 07/1963-05/2023 &   K.French & Monthly \\
			3 & {\emph{6MEOP}} & 6 &719& 07/1963-05/2023 &   K.French & Monthly \\
			4 & {\emph{100MEOP}}& 100 &719& 07/1963-05/2023&   K.French & Monthly \\
			5 &{\emph{25BEMEOP}} & 25&719 & 07/1963-05/2023  &  K.French & Monthly\\
			6 & {\emph{100MEINV}} & 99&719 & 07/1963-05/2023  &  K.French & Monthly \\
			\hline
	\end{tabular}}
\end{table}

\item[(3)] {\emph{Rolling Window Analysis.}}\ 
We follow the ``rolling-horizon" procedures in \cite{DV.GL.UR.2009,BJ.DI.DMC.2009,DV.GL.UR.2009b} to achieve out-of-sample performance comparisons of the models. Let $T$ be the length of the data set. We divide the data set into an estimation window of length $\tau$ and a prediction window of length $T-\tau$. We explain the procedure for each time in the ``rolling-horizon" strategy: First, we use the data from the 1st estimate window of length $\tau$ to obtain an optimal portfolio strategy. Next, we apply this strategy to the prediction window of length $P$ to calculate the corresponding portfolio return. Finally, we remove the first $P$ observations of the estimation window,  then we add $P$ observations from $\tau+1$ to $\tau+ P$ to the estimation window, which forms the second estimation window. This process is repeated until the end of the data set is reached. In our experiments, we set $P=1$.
\item[(4)] {\emph{Evaluation Criteria.}}\ 
Based on the above rolling window analysis, we can obtain $T-\tau$ portfolio-weight vectors for each model, i.e., $w_t$ for $t=\tau,\dots,T-1.$ Let $r_{t+1}$ be the return in the $(t+1)$th period.  We can compute out-of-sample returns for $T-\tau-1$ periods. Furthermore,
we evaluate the out-of-sample performance of each model through five performance criteria: (1)  out-of-sample portfolio average return $\widehat{\mu}$; (2)  out-of-sample portfolio variance $\widehat{\sigma}^2$; (3)  out-of-sample portfolio sharp ratio $\widehat{SR}$; (4) portfolio turnover (TURN); (5) conditional value-at-risk (CVaR).

\quad The out-of-sample portfolio average return $\widehat{\mu}$ and variance $\widehat{\sigma}^2$ are  defined by
$$
\widehat{\mu}=\frac{1}{T-\tau}\sum_{t=\tau}^{T-1}w_t^{\top}r_{t+1},\ \  \widehat{\sigma}^2 = \frac{1}{T-\tau-1}\sum_{t=\tau}^{T-1}(w_t^{\top}r_{t+1}-\widehat{\mu})^2. 
$$

Sharpe ratio \cite{SWF.1998} is used to evaluate the expected return per unit of risk, which  is given by 
$$
\widehat{SR} =\frac{ \widehat{\mu}}{\widehat{\sigma}}.
$$
Portfolio turnover represents the frequency with which assets in the portfolio are bought and sold. Let $w_{i,t}$ be the $i$-th component of $w_t$. Following \cite{AT.MN.JG.YK.2013}, the turnover of the portfolio is computed by 
$$
{\rm{TURN}}= \frac{1}{T-\tau-1}\sum_{t = \tau}^{T-1}\sum_{i=1}^{m}|w_{i,{t+1}}-\frac{1+r_{i,{t+1}}}{1+r_{t+1}^{\top}w_t}w_{i,t}|.
$$
CVaR  denotes the $95\%$-conditional value-at-risk \cite{RRT.US.2002} of portfolio losses, which can be computed by
$$
{\rm{CVaR}} = \min_{\eta}\left\lbrace \eta+ \frac{1}{(T-\tau)(1-95\%)}\sum_{t=\tau}^{T-1}\max(-w_t^{\top}r_{t+1}-\eta,0)\right\rbrace .
$$
A better investment strategy usually has higher returns, higher sharpe ratio, lower variance, turnover, and CVaR.

\item[(5)] {\emph{Parameter Setting.}}\
In the Wasserstein DR-MLSAD model, we need to choose some key parameters such as the target return $\bar{\rho}$, $\rho$ and the  radius $\epsilon$.
To ensure that problem \eqref{primalpp} is feasible, we first solve problem \eqref{primalpp} under distribution $\widehat{\mathbb{P}}_N$ without $\mathbb{E}_{\widehat{\mathbb{P}}_N}(\xi^{\top}x)=\bar{\rho}$ to get an optimal solution $x_N$, then we choose $\bar{\rho}=\mathbb{E}_{\widehat{\mathbb{P}}_N}(\xi^{\top}x_N)$. In order to guarantee the feasibility of problem \eqref{DR2}, we set $\rho = \bar{\rho}-\epsilon.$
For the choice of $\epsilon$, we apply Algorithm \ref{algepsi} to estimate a suitable radius. Another way to find $\epsilon$ is to utilize $5$-fold cross validation (CV) method.
Specifically, for CV method, we search for $\epsilon$ in the interval $[0.01,0.15]$ with step size of $0.02$, and $\epsilon$ is selected as the value that leads to the minimum risk of the portfolio. Moreover, the rolling horizon experiments were carried out  with the parameter $\tau = 90$. We repeated the experiment on these $6$ datasets to obtain convincing results.
\end{itemize}

\subsubsection{Numerical results of model performance}\label{5.1.2}
We compare the out-of-sample performance of  DR-MLSAD model with the radius selected by the RWPI approach (RWPI-DRMLSDA), DR-MLSAD model with the radius selected by the cross-validation (CV-DRMLSAD),  $1/N$ strategy and  SAA model. Furthermore, we  plot the cumulative wealth curves of different portfolios on different datasets.
\begin{table}[h]
\caption{Portfolio out-of-sample average return  $(\hat{\mu})$, variance $(\hat{\sigma}^2)$,  sharpe ratio $(\widehat{SR})$, turnover $({\rm{TURN}})$ and conditional value-at-risk $({\rm{CVaR}})$ of RWPI-DRMLSDA, CV-DRMLSAD, $1/N$ and SAA ($\tau=90$) }\label{Tab:2}%
\begin{tabular*}{\textwidth}{@{\extracolsep\fill}ccccccc}
	\hline Dataset &  Model & $\hat{\mu}$&$\hat{\sigma}^2$& $\widehat{SR}$ & ${\rm{TURN}}$& ${\rm{CVaR}} $\\
	\midrule
	10MOM	 & RWPI-DRMLSAD	 & 0.9012 	 & 20.5782 	 & 0.1987 	 & {\bf{0.5388}} 	 & {\bf{9.2394}} \\  
	& CV-DRMLSAD	 & 0.9097 	 & {\bf{20.5076}} 	 & {\bf{0.2009}} 	 & 0.6965 	 & 9.2583 \\  
	& $1/N$	 & {\bf{1.0553}} 	 & 29.7033 	 & 0.1936 	 & 0.9259 	 & 11.3068 \\  
	& SAA	 & 0.9096 	 & 20.5247 	 & 0.2008 	 & 0.7055 	 & 9.2623 \\  
	25MEOP	 & RWPI-DRMLSAD	 & {\bf{1.2015}} 	 & 22.2016 	 & {\bf{0.2550}} 	 & 0.4649 	 & 9.6203 \\  
	& CV-DRMLSAD	 & 1.1844 	 & {\bf{22.0150}} 	 & 0.2524 	 & {\bf{0.4458}} 	 & {\bf{9.5296}} \\  
	& $1/N$	 & 0.9281 	 & 33.4317 	 & 0.1605 	 & 0.9437 	 & 12.4145 \\  
	& SAA	 & 1.1855 	 & 22.0388 	 & 0.2525 	 & 0.4467 	 & 9.5369 \\  
	6MEOP	 & RWPI-DRMLSAD	 & {\bf{0.9459}} 	 & {\bf{22.3054}} 	 & {\bf{0.2003}} 	 & {\bf{0.3746}} 	 & {\bf{9.8659}} \\  
	& CV-DRMLSAD	 & 0.9270 	 & 22.3740 	 & 0.1960 	 & 0.3829 	 & 9.9155 \\  
	& $1/N$	 & 0.8066 	 & 30.6892 	 & 0.1456 	 & 2.8403 	 & 11.9777 \\  
	& SAA	 & 0.9295 	 & 22.3847 	 & 0.1965 	 & 0.3839 	 & 9.9075 \\  
	100MEOP	 & RWPI-DRMLSAD	 & {\bf{1.1878}} 	 & {\bf{21.8240}} 	 & {\bf{0.2543 	}} & {\bf{0.8953}} 	 & 9.7052 \\  
	& CV-DRMLSAD	 & 1.1645 	 & 21.8511 	 & 0.2491 	 & 0.9379 	 & {\bf{9.6608}} \\  
	& $1/N$	 & 0.9196 	 & 33.8907 	 & 0.1580 	 & 1.0531 	 & 12.4937 \\  
	& SAA	 & 1.1595 	 & 21.8864 	 & 0.2478 	 & 0.9381 	 & 9.7169 \\  
	25BEMEOP	 & RWPI-DRMLSAD	 & {\bf{0.8491}} 	 & 24.1743 	 & {\bf{0.1727}} 	 & 1.4880 	 & {\bf{11.8268}} \\  
	& CV-DRMLSAD	 & 0.7914 	 & {\bf{24.1638}} 	 & 0.1610 	 & 1.2041 	 & 12.0147 \\  
	& $1/N$	 & 0.7419 	 & 28.5108 	 & 0.1390 	 & 9.0624 	 & 11.9456 \\  
	& SAA	 & 0.7948 	 & 24.4234 	 & 0.1608 	 & {\bf{1.1529}} 	 & 12.0832 \\  
	100MEINV	 & RWPI-DRMLSAD	 & {\bf{0.9769}} 	 & 20.8309 	 & {\bf{0.2140}} 	 & {\bf{0.9593}} 	 & {\bf{9.9284}} \\  
	& CV-DRMLSAD	 & 0.9743 	 & 20.8211 	 & 0.2135 	 & 1.3241 	 & 10.0029 \\  
	& $1/N$	 & 0.9291 	 & 32.7525 	 & 0.1624 	 & 1.2149 	 & 12.3543 \\  
	& SAA	 & 0.9681 	 & {\bf{20.6880}}	 & 0.2128 	 & 1.0929 	 & 9.9827 \\
	\hline\\
\end{tabular*}
{\emph{Note.}}\ The best result in each experiment is highlighted in bold.
\end{table}

\begin{figure}[h]
	\centering
	\includegraphics[width=0.9\textwidth]{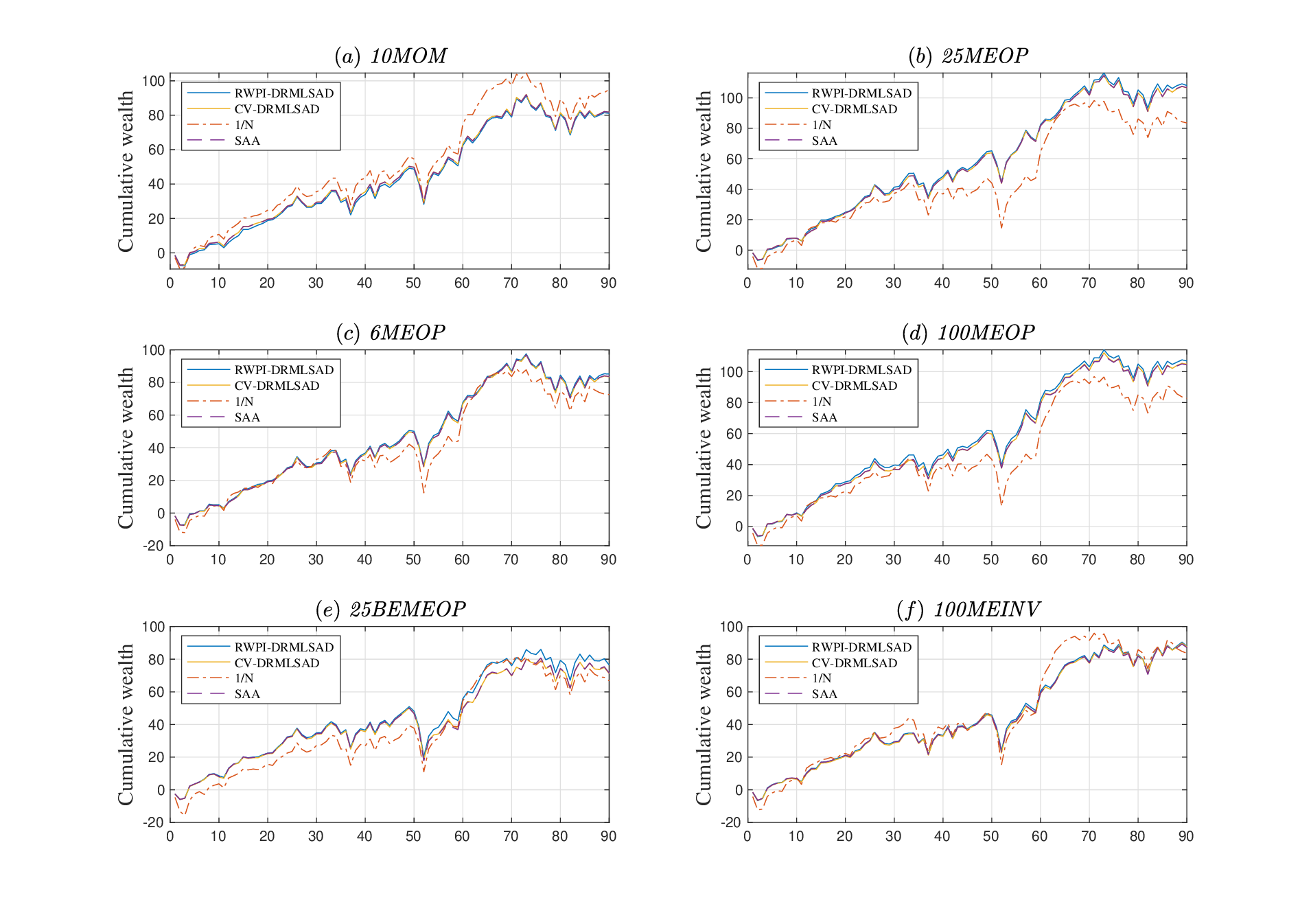}
	\caption{Portfolio out-of-sample  cumulative wealth of different portfolio selection models for six datasets}\label{Fig:1}
\end{figure}

The performance of the four strategies for Dataset 1-6 under the five criteria $\hat{\mu}$, $\hat{\sigma}^2$, $\widehat{SR}$, ${\rm{TURN}}$ and ${\rm{CVaR}}$  is shown in Table \ref{Tab:2}. The corresponding evaluations of the accumulated wealth corresponding to the four strategies  for Dateset 1-6 are displayed in Fig. \ref{Fig:1}. As shown in Table \ref{Tab:2} and Fig. \ref{Fig:1}, we make some analysis for each dataset below:
\begin{itemize}
\item For Dataset {\emph{100MOM}},  ``$\hat{\sigma}^2$" and ``$\widehat{SR}$" of CV-DRMLSAD outperform those of RWPI-DRMLSAD, $1/N$ and SAA,  ``${\rm{TURN}}$" and ``${\rm{CVaR}}$" of  RWPI-DRMLSAD  outperform those of  CV-DRMLSAD , $1/N$ and  SAA, while ``$\hat{\mu}$" of $1/N$ is better than those of other models.  Moreover, we find that all evaluation criteria of CV-DRMLSAD are better than those of SAA.
The  cumulative wealth curve of $1/N$ dominates those of other models. 
\item For Dataset {\emph{25MEOP}}, we observe that ``$\hat{\mu}$" and ``$\widehat{SR}$" of RWPI-DRMLSAD are better than of  CV-DRMLSAD, they are opposite for ``$\hat{\sigma}^2$", ``${\rm{TURN}}$" and ``${\rm{CVaR}}$". The cumulative wealth curve of RWPI-DRMLSAD dominates those of other models most of the time.
\item For Dataset {\emph{6MEOP}}, RWPI-DRMLSAD performs the best among all the evaluation criteria. Furthermore, by comparing SAA and CV-DRMLSAD, it is not difficult to find that ``$\hat{\sigma}^2$" and ``${\rm{TURN}}$" of CV-DRMLSAD are better than those of SAA, and the situation is opposite for other evaluation criteria. The cumulative wealth curve of RWPI-DRMLSAD dominates those of other models most of the time.
\item For Dataset {\emph{100MEOP}}, every evaluation criteria of the DRMLSAD model are better than that of $1/N$ and SAA. Further comparing the RWPI-DRMLSAD and CV-DRMLSAD, we can see that RWPI-DRMLSAD performs better than CV-DRMLSAD on all evaluation criteria except ``${\rm{CVaR}}".$ The cumulative wealth curve of RWPI-DRMLSAD dominates those of other models most of the time.
\item For Dataset {\emph{25BEMEOP}} and {\emph{100MEINV}},  the out-of sample performance of RWPI-DRMLSAD is more prominent than of CV-DRMLSAD  on most evaluation criteria. For  Dataset {\emph{25BEMEOP}},  the cumulative wealth curve of RWPI-DRMLSAD dominates those of other models most of the time.  For Dataset {\emph{100MEINV}}, the cumulative wealth curve of $1/N$ dominates those of other models most of the time.
\end{itemize} 

In conclusion, through the observation value of each evaluation criteria we find the out-of-sample performance of RWPI-DRMLSAD is better than other models in most cases.
The average return of RWPI-DRMLSAD is desirable in most cases, while the variance is relatively large, which can be mainly explained by high return accompanied by high risk. On the other hand, we see that sometimes the out-of-sample performance of CV-DRMLSAD may not be as good as that of SAA, which may be due to the radius found by cross-validation is not very appropriate. Based on the above analysis, we conservatively believe in the effectiveness of the RWPI approach for selecting the radius.
\subsection{Comparison of algorithms}
In this section, we perform numerical experiments to verify the efficiency of {\sc PpdSsn}  algorithm  on real data sets and random data sets.  In Section \ref{5.2.1}, we introduce a series of contrasting first-order algorithms; In Section \ref{5.2.2}, we give the parameter settings, which are crucial for the efficiency of the algorithms; In Section \ref{5.2.3}, we show the performance comparison of all algorithms for the DR-MLSAD model on random  data sets.
\subsubsection{ First-order methods}\label{5.2.1}
Inspired by \cite{HTMChu.KCTho.YJZhang.2022}, we develop a first order algorithm based on alternating direction method of multipliers (ADMM)\cite{DG.BM.1976} to solve problem \eqref{DR3} and its dual problem, respectively. Introducing the variables $\alpha\in\mathbb{R}^m,z\in\mathbb{R}^N$ and $\xi\in\mathbb{R}^m,$ we can obtain the equivalent form of problem \eqref{3.8} and its dual in minimization form as follows:
\begin{align}
&\min_{x,\alpha\in\mathbb{R}^m,y\in\mathbb{R}^N}\left\{
\frac{1}{2N}e^{\top}y+ \frac{1}{2N}\|y\|_1+\chi_{\mathcal{C}}(\alpha)
\ | \ y = Ax-\epsilon e, x-\alpha=0
\right\},\label{5.1}\\
&\min_{u,z\in\mathbb{R}^N,\xi\in\mathbb{R}^m}\left\{
\left\langle u,\epsilon e\right\rangle +\chi_{\frac{1}{2N}\mathcal{B}_{\infty}}(z)+\chi_{\mathcal{C}}^*(\xi)
\ | \ -A^{\top}u-\xi=0,u-\frac{1}{2N}e = z
\right\},\label{5.2}
\end{align}
where $\chi_{\frac{1}{2N}\mathcal{B}_{\infty}}$ and $\chi_{\mathcal{C}}^*(\cdot)$ denote the conjugate of the functions $\frac{1}{2N}\|\cdot\|_1$ and $\chi_{\mathcal{C}}(\cdot)$ , respectively.

Given a positive constant $\tilde{\rho}$, we consider the augmented Lagrangian functions for problems \eqref{5.1} and \eqref{5.2}:
$$
\begin{aligned}
&\begin{aligned}
	&\mathcal{L}_{\tilde{\rho}}^{(P)}(x,\alpha,y;u,\xi)  =  \frac{1}{2N}e^{\top}y+ \frac{1}{2N}\|y\|_1+\chi_{\mathcal{C}}(\alpha)+\frac{\tilde{\rho}}{2}\|Ax-\epsilon e-y+\frac{1}{\tilde{\rho}}\lambda\|^2-\frac{1}{2\tilde{\rho}}\|u\|^2 \\
	&+\frac{\tilde{\rho}}{2}\|x-\alpha+\frac{1}{\tilde{\rho}}\xi\|^2-\frac{1}{2\tilde{\rho}}\|\xi\|^2,\ \forall (x,\alpha,y;u,\xi)\in\mathbb{R}^m\times\mathbb{R}^m\times\mathbb{R}^N\times\mathbb{R}^N\times\mathbb{R}^m,
\end{aligned}\\
&\begin{aligned}
	&\mathcal{L}_{\tilde{\rho}}^{(D)}(u,\xi,z;x,y)  =  \chi_{\frac{1}{2N}\mathcal{B}_{\infty}}(z)
	+\frac{\tilde{\rho}}{2}\|-A^{\top}u-\xi+\frac{1}{\rho}x\|^2+\frac{\tilde{\rho}}{2}\|u-z-\frac{1}{2N}e+\frac{1}{\tilde{\rho}}y\|^2 \\
	&+ \langle u,\epsilon e\rangle+\chi_{\mathcal{C}}^*(\xi)-\frac{1}{2\tilde{\rho}}\|x\|^2-\frac{1}{2\tilde{\rho}}\|y\|^2,\ \forall (u,\xi,z;x,y) \in\mathbb{R}^N\times\mathbb{R}^m\times\mathbb{R}^N\times\mathbb{R}^m\times\mathbb{R}^N.
\end{aligned}
\end{aligned}
$$

{\emph{Alternating direction method of multipliers for primal problem \eqref{5.1}}.} ADMM for solving the primal  problem \eqref{5.1} is outlined in Algorithm \ref{Alg:3}. 
\begin{algorithm}[!h]
\caption{  (pADMM) Alternating direction method of multipliers for \eqref{5.1}} \label{Alg:3}
\begin{algorithmic}
	\State {\bf{Initialize:}}   $\tilde{\rho}>0$, $\tau=1.618$, $y^0 = {\bf{0}}\in\mathbb{R}^N,\alpha^0={\bf{0}}\in\mathbb{R}^m,u^0 = {\bf{0}}\in\mathbb{R}^N,\xi^0={\bf{0}}\in\mathbb{R}^m$. 
	\For{$k=1,2,\dots,$}
	\State {\bf{Step 1: }} Compute $x^{k+1}$ by
	\begin{equation}\label{linear2}
		\begin{aligned}
			x^{k+1}&=\mathop{\arg\min}\limits_{x\in\mathbb{R}^m}\left\{ \frac{\tilde{\rho}}{2}\|Ax-\epsilon e-y^k+\frac{1}{\tilde{\rho}}u^k\|^2+\frac{\tilde{\rho}}{2}\|x-\alpha^k+\frac{1}{\tilde{\rho}}\xi^k\|^2\right\}\\
			&=(I_m+A^{\top}A)^{-1}\left[ A^{\top}(y^k+\epsilon e-\frac{1}{\tilde{\rho}}u^k)+(\alpha^k-\frac{1}{\rho}\xi^k)\right].
		\end{aligned}
	\end{equation}
	\State {\bf{Step 2: }} Compute $y^{k+1},\alpha^{k+1}$ by
	$$
	y^{k+1}={\rm{Prox}}_{\frac{1}{2N\tilde{\rho}}\|\cdot\|_1}\left(Ax^{k+1}-\epsilon e-\frac{1}{\tilde{\rho}}(\frac{1}{2N}e-u^k)\right),\
	\alpha^{k+1} = \Pi_{\mathcal{C}}(x^{k+1}+\frac{1}{\tilde{\rho}}\xi^k).
	$$
	\State {\bf{Step 3: }} Update $u^{k+1},\xi^{k+1}$ by
	$$
	u^{k+1} = u^k+\tau\tilde{\rho}(Ax^{k+1}-\epsilon e-y^{k+1}),\ 	\xi^{k+1} = \xi^k+\tau\tilde{\rho}(x^{k+1}-\alpha^{k+1}).
	$$
	\EndFor
	\State {\bf{Output}} $x^{k+1}$ and $y^{k+1}$.
\end{algorithmic}
\end{algorithm} 
Since the computation of linear system \eqref{linear2} is the most expensive part of the Algorithm \ref{Alg:3}, we  use some effective techniques for  solving linear system to reduce the computational cost. When $m\leq N$, we use the Cholesky factorization or the conjugate gradient method to solve linear system \eqref{linear2}. Otherwise, we apply the Sherman-Morrison-Woodbury formula 
$$
(I_m+A^{\top}A)^{-1}=I_{m}-A^{\top}(I_N+AA^{\top})^{-1}A,
$$
and we calculate the inverse of the matrix $(I_N+AA^{\top})$ by Cholesky factorization.

{\emph{Alternating direction method of multipliers for dual problem \eqref{5.2}}.} The ADMM for solving the dual problem \eqref{5.2} is outlined in Algorithm \ref{Alg:4}. Similarly, we consider reducing the computational cost of solving the linear system \eqref{linear3} in Algorithm \ref{Alg:4}. When $N<m$, we utilize either the Cholesky factorization or the conjugate gradient method for solving linear system \eqref{linear3}. Otherwise, we apply the Sherman-Morrison-Woodbury formula 
$$
(I_N+AA^{\top})^{-1}=I_{N}-A(I_m+A^{\top}A)^{-1}A^{\top},
$$
then we compute the Cholesky factorization of  a smaller $m\times m$ matrix $I_m+A^{\top}A$.
\begin{algorithm}[!h]
\caption{  (dADMM) Alternating direction method of multipliers for \eqref{5.2}} \label{Alg:4}
\begin{algorithmic}
	\State {\bf{Initialize:}}  $\tilde{\rho}>0$, $\tau=1.618$, $x^0 = {\bf{0}}\in\mathbb{R}^m,y^0={\bf{0}}\in\mathbb{R}^N,z^0 = {\bf{0}}\in\mathbb{R}^N,\xi^0={\bf{0}}\in\mathbb{R}^m$. 
	\For{$k=1,2,\dots,$}
	\State {\bf{Step 1: }} Compute $u^{k+1}$ by
	\begin{equation}\label{linear3}
		\begin{aligned}
			u^{k+1}&=\mathop{\arg\min}\limits_{u\in\mathbb{R}^N}\left\{  \left\langle u,\epsilon e\right\rangle 
			+\frac{\tilde{\rho}}{2}\|-A^{\top}u-\xi^k+\frac{1}{\tilde{\rho}}x^k\|^2+\frac{\tilde{\rho}}{2}\|u-z^k-\frac{1}{2N}e+\frac{1}{\tilde{\rho}}y^k\|^2 \right\}\\
			&=(I_N+AA^{\top})^{-1}\left[ A(\frac{x^k}{\tilde{\rho}}-\xi^k)+(z^k+\frac{e}{2N}-\frac{y^k}{\tilde{\rho}}-\frac{\epsilon e}{\tilde{\rho}})\right].
		\end{aligned}
	\end{equation}
	\State {\bf{Step 2: }} Compute $\xi^{k+1},z^{k+1}$ by
	$$
	\begin{aligned}
		&\xi^{k+1}=\left({\tilde{\rho}}^{-1}x^k-A^{\top}u^{k+1}\right)-\tilde{\rho}^{-1}\Pi_{\mathcal{C}}(x^k-\rho A^{\top}u^{k+1}),\\
		&
		z^{k+1} = \Pi_{\frac{1}{2N}\mathcal{B}_{\infty}}\left(u^{k+1}-\frac{1}{2N}e+\frac{1}{\tilde{\rho}}y^k\right).
	\end{aligned}
	$$
	\State {\bf{Step 3: }} Update $x^{k+1},y^{k+1}$ by
	$$
	x^{k+1} = x^k+\tau\tilde{\rho}(-A^{\top}u^{k+1}-\xi^{k+1}),\ 	y^{k+1} = y^k+\tau\tilde{\rho}(u^{k+1}-\frac{e}{2N}-z^{k+1}).
	$$
	\EndFor
	\State {\bf{Output}} $u^{k+1}$.
\end{algorithmic}
\end{algorithm} 

\subsubsection{ Stopping criteria and parameter settings}\label{5.2.2}
We clarify the stopping criteria and set the relevant parameters in all the algorithms.

{\bf{Stopping criteria.}}
In our numerical  experiments, we use the relative KKT residual
\begin{equation*}
\begin{split}
	&Res_{1}:=\max\left(\frac{\|x-\Pi_{\mathcal{C}}(x+A^{\top}u)\|}{1+\|x\|+\|u\|},\frac{\|y-{\rm{Prox}}_{\frac{1}{2N}\|\cdot\|_1}(y-\frac{1}{2N}e-u)\|}{1+\|y\|+\|u\|}\right),\\ 
	&{Res_{2}}:=\frac{\|y-Ax+\epsilon e\|}{1+\|y\|+\|Ax\|},\ R_{kkt}:=\max\left\{Res_{1},Res_{2}\right\}
\end{split}	
\end{equation*}
to measure the accuracy of the approximate solution $(x,y,u)$ to the  system \eqref{kkt}. Let $tol$ be the tolerance. We terminate the {\sc PpdSsn} algorithm when $R_{kkt}\leq tol$ or the number of iterations exceeds $500$. We start the {\sc PpdSsn} algorithm with $(x^0,u^0)=({\bf{0}},{\bf{0}}).$

For the  algorithms pADMM and dADMM, we stop the algorithms when 
\begin{equation*}
	\begin{aligned}
R_{kkt}&=\max\left(\frac{\|x-\Pi_{\mathcal{C}}(x-A^{\top}u)\|}{1+\|x\|+\|u\|},\frac{\|y-{\rm{Prox}}_{\frac{1}{2N}\|\cdot\|_1}(y-\frac{1}{2N}e+u)\|}{1+\|y\|+\|u\|},\frac{\|Ax-\epsilon e-y\|}{1+\|y\|+\|Ax\|}\right)\\
&\leq tol,
\end{aligned}
\end{equation*}
or the number of iterations exceeds $50000$.

{\bf{Parameter settings.}} 
For the {\sc PpdSsn} in Algorithm \ref{Alg:1}, we initialize the parameter $\sigma_0=7.6,\gamma_0 = 9400.$ Then we adjust $\sigma_k$ and$\gamma_k$ adaptively by the following rule. If $Res_{1}^k\leq Res_{2}^k,$ we set $\sigma_k=\max(10^{-3},\rho_1\sigma_{k-1})$ and $\gamma_k=\min(10^6,\rho_2\gamma_{k-1})$ with 
\begin{equation*}
\rho_1=\left\{
\begin{aligned}
	&0.9, &s_2^k>s_1^k,\\
	&0.95,	&s_2^k=s_1^k,\\
	&0.98,& s_2^k<s_1^k,
\end{aligned}\right.\quad 
\rho_2=\left\{
\begin{aligned}
	&1.2, &s_2^k>s_1^k,\\
	&1.3 & s_2^k=s_1^k,\\
	&1.4,&s_2^k<s_1^k,
\end{aligned}\right.
\end{equation*}
where $s_1^k:=\frac{\|x^{k+1}-\Pi_{\mathcal{C}}(x^{k+1}+A^{\top}u^{k+1})\|}{1+\|x^{k+1}\|+\|u^{k+1}\|}$ and $s_2^k:=\frac{\|y^{k+1}-{\rm{Prox}}_{\frac{1}{2N}\|\cdot\|_1}(y^{k+1}-\frac{1}{2N}e-u^{k+1})\|}{1+\|y^{k+1}\|+\|u^{k+1}\|}$. Otherwise, we set $\sigma_k=\max(10^{-2},\rho_1\sigma_{k-1})$ and $\gamma_k=\max(2\times10^6,\rho_2\gamma_{k-1})$ with 
\begin{equation*}
\rho_1=\left\{
\begin{aligned}
	&1.02, &s_2^k<s_1^k,s_1^k>1.5Res_{2}^k,\\
	&1.05 & 1.5Res_{2}^k>s_1^k,s_1^k>s_2^k,\\
	&1.03,&s_2^k>s_1^k,s_1^k>1.1Res_{2}^k,\\
	&1.01,& {\rm{otherwise}},
\end{aligned}\right.\quad
\rho_2=\left\{
\begin{aligned}
	&1.2, &s_2^k>s_1^k, s_2^k>1.5Res_{2}^k,\\
	&1.15,	&1.5Res_{2}^k\geq s_2^k, s_2^k=s_1^k,\\
	&1.1,& s_2^k<s_1^k,s_2^k>1.1Res_{2}^k,\\
	&1.0& {\rm{otherwise}},
\end{aligned}\right.
\end{equation*}
where $Res_{1}^k$ and $Res_{2}^k$ denote the value of $Res_{1}$ and $Res_{2}$ respectively at $k$-th outer iteration.

In Algorithm \ref{Alg:2}, we set $\epsilon_j=\max\left(10^{-6},\max(10^{-8},100\|\nabla\psi(u^j)\|)\right), \bar \varrho=0.5,\vartheta=10^{-4}.$ We terminate CG algorithm at the $j$-th {\sc Ssn} when $$
\|(\mathcal{V}_j+\epsilon_j I_N)d^j+\nabla\psi(u^{j})\|\leq \min(10^{-6},\|\psi(u^{j})\|^{1.1}).
$$
In addition, we set  $\tilde{\rho}=0.01$  in Algorithm \ref{Alg:3}. We set $\tilde{\rho}=1$ in Algorithm \ref{Alg:4}.
\subsubsection{Numerical results for  data sets}\label{5.2.3}
In this section, we conduct experiments on real data sets and large-scale random data sets to demonstrate the good performance of {\sc PpdSsn} algorithm. 
\begin{itemize}
\item [(1)] {\emph{Real Data Sets.}} 	We compare the {\sc PpdSsn} algorithm, pADMM, dADMM and  Gurobi solver on the real data sets. The specific information of the real data is shown in Table \ref{Tab:1}. To solve problem \eqref{EDRMLSAD} directly, we call the Gurobi \cite{GUROBI} solver with the help of Yalmip solver. For the sake of convenience, we set the radius $\epsilon=0.15$ and the expected return $\rho = \frac{0.2}{N}\hat{\mu}^{\top}e$, where $\hat{\mu}=\frac{1}{N}\sum_{i=1}^{N}\hat{\xi}_i$. 
\begin{center}
	\setlength{\tabcolsep}{5mm}{
		\begin{longtable}{|c|c|c|c|c|}
			\captionsetup{width=0.8\textwidth}
			\caption{Numerical results of the {\sc PpdSsn}, pADMM, dADMM and Gurobi on real data set when $R_{kkt}\leq 10^{-5}$. ``32(90)" means 32 outer iterations (the total number of inner iterations is 90) and times are shown in seconds } \label{Tab:6}\\
			\hline	 \multirow{1}*{Dataset}  &\multirow{1}*{ \textbf{Algorithm}}&\multirow{1}*{\textbf{iter}}& \multicolumn{1}{c|}{\textbf{time}} & \multicolumn{1}{c|}{$R_{kkt}$} \\ \hline
			\endfirsthead	
			\hline 
			\endfoot
			\hline	 \multirow{1}*{\textbf{$(N,m)$} }  &\multirow{1}*{ \textbf{Algorithm}}&\multirow{1}*{\textbf{iter}}& \multicolumn{1}{c|}{\textbf{time}} & \multicolumn{1}{c|}{$R_{kkt}$} \\ \hline
			\endhead
			${\rm{10MOM}}$	 & PPDNA	 &  32(90) 	 & 0.64 	 & 5.7e-08\\  
			& pADMM	 &  50000 	 & 7.45 	 & 6.7e-04\\  
			& dADMM	 &  50000 	 & 11.56 	 & 1.0e-02\\  
			& Gurobi &  -- 	 & 0.30 	 & --\\  
			\hline ${\rm{25MEOP}}$	 & PPDNA	 &  95(227) 	 & 1.61 	 & 1.0e-06\\  
			& pADMM	 &  50000 	 & 9.00 	 & 2.2e-04\\  
			& dADMM	 &  50000 	 & 25.68 	 & 4.2e-03\\  
			& Gurobi	 &  -- 	 & 0.30 	 & --\\  
			\hline ${\rm{6MEOP}}$	 & PPDNA	 &  21(66) 	 & 0.52 	 & 8.0e-06\\  
			& pADMM	 &  50000 	 & 7.75 	 & 1.8e-04\\  
			& dADMM	 &  50000 	 & 10.64 	 & 3.5e-03\\  
			& Gurobi &  -- 	 & 0.30 	 & --\\  
			\hline ${\rm{100MEOP}}$	 & PPDNA	 &  16(71) 	 & 0.70 	 & 2.6e-07\\  
			& pADMM	 &  50000 	 & 32.25 	 & 6.4e-04\\  
			& dADMM	 &  50000 	 & 65.84 	 & 5.0e-02\\  
			& Gurobi	 &  --	 & 0.42 	 & --\\  
			\hline ${\rm{25BEMEOP}}$	 & PPDNA	 &  22(80) 	 & 0.69 	 & 6.0e-06\\  
			& pADMM	 &  50000 	 & 9.90 	 & 4.7e-04\\  
			& dADMM	 &  50000 	 & 26.09 	 & 1.2e-02\\  
			& Gurobi&  -- 	 & 0.34 	 & --\\  
			\hline ${\rm{100MEINV}}$	 & PPDNA	 &  14(67) 	 & 0.67 	 & 4.2e-08\\  
			& pADMM	 &  50000 	 & 30.56 	 & 3.7e-04\\  
			& dADMM	 &  50000 	 & 65.44 	 & 1.0e-02\\  
			& Gurobi &  -- 	 & 0.42 	 & --\\  
			\hline
	\end{longtable}}
\end{center}
\quad  Table \ref{Tab:6} shows the numerical results of the {\sc PpdSsn} algorithm, pADMM, dADMM and Gurobi solver under $tol=10^{-5}$.  This table contains the number of iterations of the algorithms, CPU time and $R_{kkt}$. We can see that the time of the Gurobi solver is slightly faster than that of the {\sc PpdSsn} algorithm. Moreover, it is easy to observe that the ADMM cannot successfully solve the real data set under $tol=10^{-5}$.

\item [(2)] {\emph{Random Data Sets.}} 
Following the approach in \cite{EPM.KD.2018}, we assume that return $\hat{\xi}_i$ can be decomposed into systemic risk factor $\varphi\sim N(0,2\%)$ common to all assets and an unsystematic or idiosyncratic risk factor $\zeta_i\sim N(i\times 3\%,i\times 2.5\%)$ specific to asset $i$, i.e., $\hat{\xi}_i=\varphi+\zeta_i,$ where $\varphi$ and the idiosyncratic risk factors $\zeta_i,\ i=1,\dots,m$ constitute independent normal random variables.  We set  $(N,m)=(200i,4000i),i=1,...,9$, where $N$ and $m$ denote the number of samples and assets, respectively. The choice of $\hat{\mu}$, $\epsilon$ and $\rho$ on the random data set is the same as on the real data set.


\quad The numerical results of  {\sc PpdSsn} algorithm, pADMM, dADMM and Gurobi under ${\rm{tol}}=10^{-5}$ are listed in Table \ref{Tab:4}, which includes the number of iterations, CPU time and $R_{kkt}$. In particular, the number of iterations of the {\sc PpdSsn} algorithm in Table \ref{Tab:4} contains the number of outer and inner iterations. It is clear to observe in Table \ref{Tab:4} that all the algorithms successfully solve all the  cases. Specifically, the time of  the {\sc PpdSsn} algorithm is faster than that of both pADMM and dADMM. For large-scale data sets, the time of the {\sc PpdSsn} algorithm is about $2$ times faster than that of pADMM and  about $13$ times faster than of dADMM. Furthermore, the time of the {\sc PpdSsn} algorithm is more than  $100$ times faster than that of Gurobi.

\begin{center}
	\setlength{\tabcolsep}{6mm}{
		\begin{longtable}{|c|c|c|c|c|}
			\captionsetup{width=0.8\textwidth}
			\caption{Numerical results of the {\sc PpdSsn}, pADMM, dADMM and Gurobi on random data set when $R_{kkt}\leq 10^{-5}$. ``2(5)" means 2 outer iterations (the total number of inner iterations is 5) and times are shown in seconds } \label{Tab:4}\\
			\hline	 \multirow{1}*{\textbf{$(N,m)$} }  &\multirow{1}*{ \textbf{Algorithm}}&\multirow{1}*{\textbf{iter}}& \multicolumn{1}{c|}{\textbf{time}} & \multicolumn{1}{c|}{$R_{kkt}$} \\ \hline
			\endfirsthead	
			\hline 
			\endfoot
			\hline	 \multirow{1}*{\textbf{$(N,m)$} }  &\multirow{1}*{ \textbf{Algorithm}}&\multirow{1}*{\textbf{iter}}& \multicolumn{1}{c|}{\textbf{time}} & \multicolumn{1}{c|}{$R_{kkt}$} \\ \hline
			\endhead
			$(200,4000)$	 & PPDNA	 &  2(5) 	 & 0.06 	 & 4.2e-06\\  
			& pADMM	 &  18 	 & 0.03 	 & 6.2e-06\\  
			& dADMM	 &  34 	 & 0.66 	 & 6.2e-06\\  
			& Gurobi &  -- 	 & 2.14 	 & --\\  
			\hline $(400,8000)$	 & PPDNA	 &  2(5) 	 & 0.22 	 & 1.2e-06\\  
			& pADMM	 &  16 	 & 0.17 	 & 9.3e-06\\  
			& dADMM	 &  35 	 & 2.59 	 & 6.3e-06\\  
			& Gurobi &  -- 	 & 9.25 	 & --\\  
			\hline $(600,12000)$	 & PPDNA	 &  2(5) 	 & 0.47 	 & 6.4e-07\\  
			& pADMM	 &  16 	 & 0.61 	 & 8.1e-06\\  
			& dADMM	 &  35 	 & 6.50 	 & 8.1e-06\\  
			& Gurobi	 &  -- 	 & 21.61 	 & --\\  
			\hline $(800,16000)$	 & PPDNA	 &  2(5) 	 & 0.84 	 & 9.1e-07\\  
			& pADMM	 &  16 	 & 1.19 	 & 7.4e-06\\  
			& dADMM	 &  35 	 & 11.26 	 & 9.3e-06\\  
			& Gurobi&  -- 	 & 49.65 	 & --\\  
			\hline $(1000,20000)$	 & PPDNA	 &  2(5) 	 & 1.28 	 & 1.1e-06\\  
			& pADMM	 &  15 	 & 1.69 	 & 9.1e-06\\  
			& dADMM	 &  36 	 & 16.77 	 & 6.9e-06\\  
			& Gurobi	 &  -- 	 & 88.66 	 & --\\  
			\hline $(1200,24000)$	 & PPDNA	 &  2(5) 	 & 1.81 	 & 1.2e-06\\  
			& pADMM	 &  15 	 & 2.88 	 & 9.5e-06\\  
			& dADMM	 &  36 	 & 23.89 	 & 7.5e-06\\  
			& Gurobi &  --	 & 108.97 	 & --\\  
			\hline $(1400,28000)$	 & PPDNA	 &  2(5) 	 & 2.44 	 & 1.2e-06\\  
			& pADMM	 &  15 	 & 4.16 	 & 9.2e-06\\  
			& dADMM	 &  36 	 & 31.82 	 & 7.9e-06\\  
			& Gurobi	 &  -- 	 & 207.38 	 & --\\  
			\hline $(1600,32000)$	 & PPDNA	 &  2(5) 	 & 3.16 	 & 1.4e-06\\  
			& pADMM	 &  15 	 & 5.86 	 & 8.1e-06\\  
			& dADMM	 &  36 	 & 40.99 	 & 8.7e-06\\  
			& Gurobi	 &  -- 	 & 321.94 	 & --\\  
			\hline $(1800,36000)$	 & PPDNA	 &  2(5) 	 & 3.94 	 & 1.2e-06\\  
			& pADMM	 &  15 	 & 7.55 	 & 7.7e-06\\  
			& dADMM	 &  36 	 & 52.24 	 & 9.2e-06\\  
			& Gurobi	 & -- 	 & 524.73 	 & --\\  
			\hline
	\end{longtable}}
\end{center}
\end{itemize}

Based on the numerical results, we find that the time of  Gurobi solver is a little faster than that of the {\sc PpdSsn} algorithm when solving  real data sets with a small scale, but once the scale of the problem is larger, the efficiency of the {\sc PpdSsn} algorithm is significantly higher than that of the Gurobi solver. At the same time, although the first-order algorithm can successfully solve large-scale random data sets, it can not solve small-scale real data under $tol = 10^{-5}$, which indicate that ADMM is not very robust in solving such problems.
Therefore, we can safely conclude that {\sc PpdSsn} algorithm is more efficient than pADMM, dADMM and Gurobi solver for solving large-scale distributionally robust mean-lower semi-absolute deviation model.

\section{Conclusion}\label{sec:6}
In this work, we considered a distributionally robust mean-lower semi-absolute deviation portfolio model based on the Wasserstein metric. We have transformed it into an unconstrained convex optimization problem. Since the size of the ambiguity set is a crucial parameter in the model, we have developed a RWPI approach to find the appropriate radius. Numerical results have shown that the DR-MLSAD model with the radius selected by the RWPI approach has good out-of-sample performance. For the transformed convex problem, we have designed an efficient {\sc PpdSsn} algorithm for solving it and reduced the computational cost by exploring some of its special structures. Experimental results have verified that {\sc PpdSsn} algorithm is more efficient than pADMM, dADMM and Gurobi solver  in solving large-scale DR-MLSAD models.

\bmhead{Acknowledgements}

The work of Yong-Jin Liu was in part supported by the National Natural Science Foundation of China (Grant No. 12271097), the Key Program of National Science Foundation of Fujian Province of China (Grant No. 2023J02007), and the Fujian Alliance of Mathematics (Grant No. 2023SXLMMS01).

\section*{Declarations}
\begin{itemize}
\item {\small{\noindent \textbf{Conflict of interest} The authors declare that they have no conflict of interest.}}
\item  {\small{\noindent \textbf{Data Availability} All data generated or analyzed during this study are included in this  article. The data that support the findings of this study are openly available in Ken French' s website \url{https://mba.tuck.dartmouth.edu/pages/faculty/ken.french/data_library.html}}}
\end{itemize}

\begin{appendices}
	
	\section{Proof of Theorem \ref{thm2}}\label{App1}
	
	\begin{proposition}\label{fesi}
		Assume that $q$ satisfies $1/p+1/q=1$ and $q\geq 1$. Then it holds that
		$$
		\min_{\mathbb{P}\in\mathcal{U}_{\epsilon}(\widehat{\mathbb{P}}_N)} \mathbb{E}_{\mathbb{P}}(\xi^{\top}x) =\mathbb{E}_{\widehat{\mathbb{P}}_N}(\xi^{\top}x)-\epsilon \|x\|_q.
		$$
	\end{proposition}
	\begin{proof}
		It is easy to see that
		$$
		\min_{\mathbb{P}\in\mathcal{U}_{\epsilon}(\widehat{\mathbb{P}}_N)} \mathbb{E}_{\mathbb{P}}(\xi^{\top}x)=	-\max_{\mathbb{P}\in\mathcal{U}_{\epsilon}(\widehat{\mathbb{P}}_N)} \mathbb{E}_{\mathbb{P}}(-\xi^{\top}x).
		$$
		Combining with the work of Blanchet  et al. \cite{BJ.KY.MK.2016} and Slater condition,
		one has 
		\begin{equation}\label{A1}\tag{A.1}
			\max_{\mathbb{P}\in\mathcal{U}_{\epsilon}(\widehat{\mathbb{P}}_N)}\ \mathbb{E}_{\mathbb{P}}(-\xi^{\top}x)=\min_{\lambda\geq0}\left\lbrace \lambda\epsilon+\frac{1}{N}\sum_{i=1}^{N}\Phi_{\lambda}(\hat{\xi}_i)\right\rbrace,
		\end{equation}
		where 
		$$
		\begin{aligned}
			\Phi_{\lambda}(\hat{\xi}_i)&=\sup_{u}\left\lbrace -x^{\top}u-\lambda\|u-\hat{\xi}_i\|_p\right\rbrace \\
			&=\sup_{\Delta}\left\lbrace (-x^{\top})(\Delta+\hat{\xi}_i)-\lambda\|\Delta\|_p\right\rbrace \\
			&=\sup_{\Delta}\left\lbrace \|x^{\top}\|_q\|\Delta\|_p-\lambda\|\Delta\|_p\right\rbrace -x^{\top}\hat{\xi}_i\\
			&=	\left\{
			\begin{aligned}
				&-x^{\top}\hat{\xi}_i, \ {\rm{if}}\ \|x\|_q\leq \lambda,\\
				&+\infty,\ {\rm{otherwise}}.
			\end{aligned}
			\right.
		\end{aligned}
		$$
		Thus, \eqref{A1} becomes 
		$$
		\max_{\mathbb{P}\in\mathcal{U}_{\epsilon}(\widehat{\mathbb{P}}_N)}\ \mathbb{E}_{\mathbb{P}}(-\xi^{\top}x)=\min_{\|x\|_q\leq \lambda}\left\lbrace \lambda\epsilon-\frac{1}{N}\sum_{i=1}^{N}(\hat{\xi}_i)^{\top}x\right\rbrace=\epsilon\|x\|_q-\mathbb{E}_{\widehat{\mathbb{P}}_N}(\xi^{\top}x).
		$$
		Then we have 
		$$
		\min_{\mathbb{P}\in\mathcal{U}_{\epsilon}(\widehat{\mathbb{P}}_N)} \mathbb{E}_{\mathbb{P}}(\xi^{\top}x)=\mathbb{E}_{\widehat{\mathbb{P}}_N}(\xi^{\top}x)-\epsilon\|x\|_q.
		$$
	\end{proof}
	As a result, the feasible region $ \mathcal{F}_{\epsilon,\rho}$ admits the following form:
	$$
	\mathcal{F}_{\epsilon,\rho} =\{x \mid e^{\top}x=1,x\geq0,\mathbb{E}_{\widehat{\mathbb{P}}_N}(\xi^{\top}x)-\epsilon \|x\|_q\geq \rho\}.
	$$
	Clearly, $\mathcal{F}_{\epsilon,\rho} $ is a convex set.
	
	Secondly, we consider the inner maximization part of problem \eqref{EDRMLSAD}:
	\begin{equation}\label{inner}\tag{A.2}
		\begin{aligned}
			&\max_{\mathbb{P}}\ \mathbb{E}_{\mathbb{P}}\left[ \max(0,\alpha-\xi^{\top}x)\right] \\
			&\ {\rm{s.t.}}\quad \mathbb{E}_{\mathbb{P}}(\xi^{\top}x)=\alpha,  \mathbb{P}\in\mathcal{U}_{\epsilon}(\widehat{\mathbb{P}}_N). 
		\end{aligned}
	\end{equation}
	The problem \eqref{inner} is an infinite-dimensional non-convex problem. Following the work in \cite{EPM.KD.2018}, we shall prove that problem \eqref{inner} can be reformulated as a finite-dimensional conic linear programming problem.
	\begin{theorem}\label{thm1}
		Let the uncertainty set $\Xi=\mathbb{R}^m$. Then the optimal value of problem \eqref{inner} is equal to the optimal value of the following convex programming problem:
		\begin{equation}\label{conic}\tag{A.3}
			\begin{aligned}
				&\min_{ \gamma_1,\gamma_2}\ \gamma_1(\alpha-\hat{\mu}^{\top}x)+\gamma_2\epsilon+\frac{1}{N}\sum_{i=1}^{N}\max(0,\alpha-\hat{\xi}_i^{\top}x) \\
				&\quad\ \ {\rm{s.t.}}\quad \|\gamma_1x\|_{q}\leq\gamma_2,\\
				&\quad\ \ \quad \quad\  \|\gamma_1x+x\|_{q}\leq \gamma_2,
			\end{aligned}
		\end{equation}
		where $\hat{\mu}=\frac{1}{N}\sum_{i=1}^{N}\hat{\xi}_i$.
	\end{theorem}
	
	\begin{proof}
		By the definition of Wasserstein ambiguity set, we can reformulate Problem \eqref{inner} in the following form:
		\begin{equation}\label{3.1}\tag{A.4}
			\left\{
			\begin{aligned}
				&\max_{\mathbb{P}\in\mathcal{U}_{\epsilon}(\mathbb{P}_N)}\int_{\Xi}\max(0,\alpha-\xi^{\top}x)\ \mathbb{P}(d\xi) \\
				\ &\quad {\rm{s.t.}}\quad \int_{\Xi}\xi^{\top}x\ \mathbb{P}(d\xi)=\alpha.
			\end{aligned}
			\right.=\left\{
			\begin{aligned}
				&\max_{\mathbb{Q}\in\mathcal{M}(\Xi)}\int_{\Xi}\max(0,\alpha-\xi^{\top}x)\ \mathbb{Q}(d\xi,\Xi)\\
				\ &\quad {\rm{s.t.}}\quad \int_{\Xi}\xi^{\top}x\ \mathbb{Q}(d\xi,\Xi)=\alpha,\\
				&	\quad\quad\quad \int_{\Xi^2} \|\xi-\xi'\|_p\mathbb{Q}(d\xi,d\xi')\leq \epsilon,\\
				&	\quad\quad\quad  \mathbb{Q}(\Xi,d\xi')=\mathbb{P}_N(d\xi'),\ i = 1,\dots,N.
			\end{aligned}
			\right.	
		\end{equation}
		Since $\mathbb{P}_N$ is  discrete, we have 
		\begin{equation}\label{3.2}\tag{A.5}
			\mathbb{Q}(d\xi,\Xi) = \int_{\xi'\in\Xi}\mathbb{Q}(d\xi,d\xi')=\sum_{i=1}^{N}\mathbb{Q}(d\xi|\xi'=\hat{\xi}_i)\mathbb{P}_N(\hat{\xi}_i)=\frac{1}{N}\sum_{i=1}^{N}\mathbb{P}^i(d\xi),
		\end{equation}  
		where $\mathbb{P}^i(d\xi)$ denotes the conditional probability distribution of $\xi$ given $\xi'=\hat{\xi}_i$. Similarly, it holds
		\begin{equation}\label{3.3}\tag{A.6}
			\mathbb{Q}(d\xi,d\xi')=\mathbb{Q}(d\xi|\xi')\mathbb{P}_N(\xi')=\frac{1}{N}\sum_{i=1}^{N}\delta_{\hat{\xi}_i}(\xi')\mathbb{P}^i(d\xi).
		\end{equation}
		Invoking \eqref{3.2} and \eqref{3.3}, problem \eqref{3.1}  can be equivalently written as follows:
		\begin{equation}\label{3.4}\tag{A.7}
			\left\{
			\begin{aligned}
				&\max_{\mathbb{P}^i\in\mathcal{M}(\Xi)}\frac{1}{N}\sum_{i=1}^{N}\int_{\Xi}\max(0,\alpha-\xi^{\top}x)\mathbb{P}^i(d\xi)\\
				\ &\quad {\rm{s.t.}}\quad \frac{1}{N}\sum_{i=1}^{N}\int_{\Xi}\xi^{\top}x \mathbb{P}^i(d\xi)=\alpha,\\
				&	\quad\quad\quad \frac{1}{N}\sum_{i=1}^{N}\int_{\Xi} \|\xi-\hat{\xi}_i\|_p\mathbb{P}^i(d\xi)\leq \epsilon,\\
				&	\quad\quad\quad  \int_{\Xi}\mathbb{P}^i(d\xi)=1,\ i = 1,\dots,N.
			\end{aligned}
			\right.
		\end{equation}
		Then the Lagrangian  function admits the following minimization form:
		\begin{equation*}
			\begin{aligned}
				L(\xi;s,\gamma_1,\gamma_2)& =  \frac{1}{N}\sum_{i=1}^{N}\int_{\Xi}\max(0,\alpha-\xi^{\top}x) \mathbb{P}^i(d\xi)+\gamma_1(\alpha-\frac{1}{N}\sum_{i=1}^{N}\int_{\Xi}\xi^{\top}x \mathbb{P}^i(d\xi))\\
				&\ +\gamma_2(\epsilon-\frac{1}{N}\sum_{i=1}^{N}\int_{\Xi} \|\xi-\hat{\xi}_i\|_p\mathbb{P}^i(d\xi))+\sum_{i=1}^{N}s_i(1- \int_{\Xi}\mathbb{P}^i(d\xi))\\
				&=\frac{1}{N}\sum_{i=1}^{N}\int_{\Xi}\left(  \max(0,\alpha-\xi^{\top}x)-\gamma_1\xi^{\top}x-\gamma_2\|\xi-\hat{\xi}_i\|_p-Ns_i\right)  \mathbb{P}^i(d\xi)\\
				&\quad +\gamma_1\alpha+\gamma_2\epsilon+\sum_{i=1}^{N}s_i,
			\end{aligned}
		\end{equation*}
		where $\gamma_1\in\mathbb{R},\gamma_2\in\mathbb{R}_{+},s\in\mathbb{R}^N.$ Thus,  the Lagrangian dual of problem \eqref{3.4} is 
		$$
		\min_{\gamma_1,\gamma_2,s}\max_{\xi\in\Xi}\ L(\xi;s,\gamma_1,\gamma_2),
		$$
		which is given by
		\begin{equation}\label{dual}\tag{A.8}
			\begin{aligned}
				&\min_{\gamma_1,\gamma_2,s}\ \gamma_1\alpha+\gamma_2\epsilon+\sum_{i=1}^{N}s_i\\
				\ &\quad {\rm{s.t.}}\quad \max(0,\alpha-\xi^{\top}x)-\gamma_1\xi^{\top}x-\gamma_2\|\xi-\hat{\xi}_i\|_p-Ns_i\leq 0,\ \forall \xi\in\Xi,\ i = 1,\dots,N,\\
				\ &\quad \quad\quad \ \gamma_2\geq 0.
			\end{aligned}
		\end{equation}
		For any $\epsilon>0$, it is not difficult to observe that the $\mathbb{Q}_0=\mathbb{P}_N\times\mathbb{P}_N$ is a strictly feasible solution of problem \eqref{3.1}. Therefore,  the strong dual result holds due to the Slater condition of problem \eqref{3.1} being satisfied \cite{AS.2001}. On the other hand, if $\epsilon=0$, then the duality gap between problem \eqref{3.1} and problem \eqref{dual} is also equal to $0$. In fact, for $\epsilon=0$, problem \eqref{3.1} reduces to the SAA problem \eqref{saa} with an optimal objective function value of $\frac{1}{N}\sum_{i=1}^{N}\max(0,\alpha-\hat{\xi}_i^{\top}x),$ and $\gamma_2$ in \eqref{dual} can be increased to infinity without penalty in the objective function. Following \cite{EPM.KD.2018,WZ.YK.SS.ZY.2020}, one can conclude that as $\lambda$ tends to infinity, the problem \eqref{dual} converges to the sample average $\frac{1}{N}\sum_{i=1}^{N}\max(0,\alpha-\hat{\xi}_i^{\top}x)$. As a result, the optimal value of problem\eqref{3.1} and \eqref{dual} are equal. Moreover, we reformulate \eqref{dual} as 
		\begin{equation*}
			\begin{aligned}
				&\left\{
				\begin{aligned}
					&\min_{\gamma_1,\gamma_2,s}\ \gamma_1\alpha+\gamma_2\epsilon+\sum_{i=1}^{N}s_i\\
					\ &\quad {\rm{s.t.}}\quad
					Ns_i\geq\max_{\xi\in\Xi}(-\gamma_2\|\xi-\hat{\xi}_i\|_p-\gamma_1\xi^{\top}x),\ i = 1,\dots,N,\\
					\ &\quad \quad\quad \ Ns_i\geq\max_{\xi\in\Xi}(-\gamma_2\|\xi-\hat{\xi}_i\|_p-\gamma_1\xi^{\top}x-\xi^{\top}x+\alpha),\ i = 1,\dots,N,\\
					\ &\quad \quad\quad \ \gamma_2\geq0.
				\end{aligned}
				\right.\\
				&=\left\{
				\begin{aligned}
					&\min_{\gamma_1,\gamma_2,s}\ \gamma_1\alpha+\gamma_2\epsilon+\sum_{i=1}^{N}s_i\\
					\ &\quad {\rm{s.t.}}\quad
					Ns_i\geq\max_{\xi\in\Xi}\min_{\|z_{i1}\|_{q}\leq\gamma_2}(-z_{i1}^{\top}(\xi-\hat{\xi}_i)-\gamma_1\xi^{\top}x),\ i = 1,\dots,N,\\
					\ &\quad \quad\quad \ Ns_i\geq\max_{\xi\in\Xi}\min_{\|z_{i2}\|_{q}\leq\gamma_2}(-z_{i2}^{\top}(\xi-\hat{\xi}_i)-\gamma_1\xi^{\top}x-\xi^{\top}x+\alpha),\ i = 1,\dots,N,\\
					\ &\quad \quad\quad \ \gamma_2\geq0.
				\end{aligned}
				\right.\\
				&=\left\{
				\begin{aligned}
					&\min_{\gamma_1,\gamma_2,s}\ \gamma_1\alpha+\gamma_2\epsilon+\sum_{i=1}^{N}s_i\\
					\ &\quad {\rm{s.t.}}\quad
					Ns_i\geq\max_{\xi\in\Xi}(-z_{i1}^{\top}(\xi-\hat{\xi}_i)-\gamma_1\xi^{\top}x),\ i = 1,\dots,N,\\
					\ &\quad \quad\quad \ Ns_i\geq\max_{\xi\in\Xi}(-z_{i2}^{\top}(\xi-\hat{\xi}_i)-\gamma_1\xi^{\top}x-\xi^{\top}x+\alpha),\ i = 1,\dots,N,\\
					\ &\quad\quad\quad\|z_{i1}\|_{q}\leq\gamma_2,\ \|z_{i2}\|_{q}\leq\gamma_2.
				\end{aligned}
				\right.
			\end{aligned}
		\end{equation*}
		where $\|\cdot\|_{q}$ denotes the dual norm of $\|\cdot\|_p$. Since the set $\{z_{ik}\in\mathbb{R}^N:\|z_{ik}\|_{q}\leq\gamma_2\}$ is compact for any $\gamma_2$, and by virtue of the classical maximin theorem\cite[Proposition 5.5.4]{BDP.2009}, the last equality holds. Furthermore, by eliminating $z_{i1}$ and $z_{i2}$, we obtain
		\begin{equation}\label{3.5}\tag{A.9}
			\left\{
			\begin{aligned}
				&\min_{\gamma_1,\gamma_2,s}\ \gamma_1\alpha+\gamma_2\epsilon+\sum_{i=1}^{N}s_i\\
				\ &\quad {\rm{s.t.}}\quad
				Ns_i\geq-\gamma_1\hat{\xi}_i^{\top}x,\ i = 1,\dots,N,\\
				\ &\quad \quad\quad \ Ns_i\geq(-x-\gamma_1x)^{\top}\hat{\xi}_i+\alpha,\ i = 1,\dots,N,\\
				\ &\quad\quad\quad\|\gamma_1x\|_{q}\leq\gamma_2,\ \|x+\gamma_1x\|_{q}\leq\gamma_2.
			\end{aligned}
			\right.
		\end{equation}
		The first and second constraints in \eqref{3.5} can be equivalent written as 
		$$
		s_i\geq\frac{1}{N}\left\lbrace \max(\alpha-\hat{\xi}_i^{\top}x,0)-\gamma_1x^{\top}\hat{\xi}_i\right\rbrace.
		$$
		Thus, problem \eqref{3.5} is equivalent to 
		\begin{equation*}
			\begin{aligned}
				&\min_{ \gamma_1,\gamma_2}\ \gamma_1(\alpha-\hat{\mu}^{\top}x)+\gamma_2\epsilon+\frac{1}{N}\sum_{i=1}^{N}\max(0,\alpha-\hat{\xi}_i^{\top}x) \\
				&\quad\ \ {\rm{s.t.}}\quad \|\gamma_1x\|_{q}\leq\gamma_2,\\
				&\quad\ \ \quad \quad\  \|\gamma_1x+x\|_{q}\leq \gamma_2,
			\end{aligned}
		\end{equation*}
		where $\hat{\mu}=\frac{1}{N}\sum_{i=1}^{N}\hat{\xi}_i.$ Here, we complete the proof.  
	\end{proof}
	Since the strong duality holds, the duality gap between problem \eqref{conic} and its dual problem is $0$. Therefore, we can obtain the optimal value of problem \eqref{conic} by solving its dual problem. Then we put its optimal value into model \eqref{EDRMLSAD} and simplify it further. Finally, we can obtain a finitely convex optimization problem equivalent to the DR-MLSAD model. The specific details are presented in the following theorem.
	Now, we are ready to prove Theorem \ref{thm2}.
	
	
	\noindent{ \large \bf{Proof of Theorem \ref{thm2}.}}
	Problem \eqref{conic} can be transformed into the  form below:
	\begin{equation}\label{3.6}\tag{A.10}
		\begin{aligned}
			&\min_{ \gamma_1,\gamma_2,w_1,w_2}\ \gamma_1(\alpha-\hat{\mu}^{\top}x)+\gamma_2\epsilon+\frac{1}{N}\sum_{i=1}^{N}\max(0,\alpha-\hat{\xi}_i^{\top}x) \\
			&\quad {\rm{s.t.}}\quad \|w_1\|_{q}\leq\gamma_2,\ \|w_2\|_{q}\leq \gamma_2\\
			&\quad\quad \quad\ w_1=\gamma_1x,\ w_2=\gamma_1x+x.\\
		\end{aligned}
	\end{equation}
	The corresponding Lagrangian function takes the following form:
	$$
	\begin{aligned}
		L(\gamma_1,\gamma_2,w_1,w_2;\varpi,\upsilon)&=\gamma_1(\alpha-\hat{\mu}^{\top}x)+\gamma_2\epsilon+\frac{1}{N}\sum_{i=1}^{N}\max(0,\alpha-\hat{\xi}_i^{\top}x)+\varpi_1(\|w_1\|_{q}-\gamma_2)\\
		&+\varpi_2(\|w_2\|_{q}-\gamma_2)+\upsilon_1(\gamma_1x-w_1)+\upsilon_2(\gamma_1x+x-w_2)\\
		&=\frac{1}{N}\sum_{i=1}^{N}\max(0,\alpha-\hat{\xi}_i^{\top}x)+\gamma_1(\alpha-\hat{\mu}^{\top}x+(\upsilon_1+\upsilon_2)^{\top}x)+\upsilon_2^{\top}x\\
		&+(\varpi_1\|w_1\|_{q}-\upsilon_1^{\top}w_1)+(\varpi_2\|w_2\|_{q}-\upsilon_2^{\top}w_2)+\gamma_2(\epsilon-\varpi_1-\varpi_2)
	\end{aligned}
	$$
	where $\varpi_1,\varpi_2\in\mathbb{R}_{+}$ and $\upsilon_1,\upsilon_2\in\mathbb{R}^m$ are Lagrange multipliers. Then the dual of \eqref{3.6} admits the following maximization form:
	\begin{equation}\label{dual2}\tag{A.11}
		\begin{aligned}
			&\max_{\varpi_1,\varpi_2,\upsilon_1,\upsilon_2}\ \frac{1}{N}\sum_{i=1}^{N}\max(0,\alpha-\hat{\xi}_i^{\top}x)+\upsilon_2^{\top}x\\
			&\quad {\rm{s.t.}}\quad  \alpha-\hat{\mu}^{\top}x+(\upsilon_1+\upsilon_2)^{\top}x=0,\\
			&\quad \quad \quad \varpi_1+\varpi_2\leq\epsilon,\\
			&\quad \quad \quad \|\upsilon_1\|_p\leq\varpi_1,\ \|\upsilon_2\|_p\leq\varpi_2.
		\end{aligned}
	\end{equation}
	It is not difficult to find that $\gamma_1=0,\gamma_2=2,w_1=0,w_2=x$ is a strictly feasible solution to the problem \eqref{3.6}. Therefore, the Slater condition is satisfied, that is, the strong duality holds.
	Considering the norm $\|\cdot\|_p=\|\cdot\|_{\infty}$, we reformulate the problem \eqref{dual2} as
	\begin{equation}\label{3.7}\tag{A.12}
		\begin{aligned}
			&\max_{\varpi_1,\varpi_2,\upsilon_1,\upsilon_2}\ \frac{1}{N}\sum_{i=1}^{N}\max(0,\alpha-\hat{\xi}_i^{\top}x)+\upsilon_2^{\top}x\\
			&\quad {\rm{s.t.}}\quad  \alpha-\hat{\mu}^{\top}x+(\upsilon_1+\upsilon_2)^{\top}x=0,\ \varpi_1+\varpi_2\leq\epsilon,\\
			&\quad \quad \quad -\varpi_1e\leq\upsilon_1\leq\varpi_1e,\\ 
			&\quad\quad\quad\ -\varpi_2e\leq\upsilon_2\leq\varpi_2e.
		\end{aligned}
	\end{equation}
	Since the portfolio weights $x$ satisfy $e^{\top}x=1$ and $x\geq 0$, by the second and fourth constraints in \eqref{3.7}, we have
	$$
	\upsilon_2^{\top}x\leq \varpi_1e^{\top}x\leq\epsilon,
	$$
	which indicates that the optimal value of the problem \eqref{3.7} does not exceed $ \frac{1}{N}\sum_{i=1}^{N}\max(0,\alpha-\hat{\xi}_i^{\top}x)+\epsilon$.
	Thus, as long as the problem \eqref{3.7} is feasible, there will be an optimal solution so that the maximum value of problem \eqref{3.7} is $ \frac{1}{N}\sum_{i=1}^{N}\max(0,\alpha-\hat{\xi}_i^{\top}x)+\epsilon$. since $-\epsilon\leq(\upsilon_1+\upsilon_2)^{\top}x\leq (\varpi_1+\varpi_2)\leq\epsilon,$ we have $
	-\epsilon\leq \hat{\mu}^{\top}x-\alpha\leq\epsilon.$  Therefore, if $\hat{\mu}^{\top}x-\alpha\in\left[ -\epsilon,\epsilon\right]$, the problem \eqref{3.7} is feasible. Indeed, it is not difficult to see that $\varpi_1^*=0,\varpi_2^*=\epsilon,\upsilon_1^*=0,\upsilon_2^*=\varpi_2^*e$ is the unique optimal solution of the problem. In this case, the constraint $\hat{\mu}^{\top}x-\alpha=\epsilon$ holds. Then, put the above results into problem \eqref{EDRMLSAD}, DR-MLSAD model can be formulated as follows
	\begin{equation}\label{maxmax}\tag{A.13}
		\begin{aligned}
			&\min_{x\in	\mathcal{F}_{\epsilon,\rho}}\max_{\alpha}\ \frac{1}{N}\sum_{i=1}^{N}\max(\alpha-\hat{\xi}_i^{\top}x,0)+\epsilon\\
			&\quad {\rm{s.t.}}\quad \hat{\mu}^{\top}x-\alpha=\epsilon,\alpha\geq\rho.
		\end{aligned}
	\end{equation}  
	By eliminating $\alpha$, problem \eqref{maxmax} reduces to   
	$$
	\begin{aligned}
		&\left\{
		\begin{aligned}
			&\min_{x\in	\mathcal{F}_{\epsilon,\rho}}\ \frac{1}{N}\sum_{i=1}^{N}\max(\hat{\mu}^{\top}x-\hat{\xi}_i^{\top}x-\epsilon,0)+\epsilon\\
			&\quad {\rm{s.t.}}\quad \hat{\mu}^{\top}x-\epsilon\geq\rho.
		\end{aligned}
		\right.\\
		&=\left\{\begin{aligned}
			&\min_{x\in	\mathcal{F}_{\epsilon,\rho}}\ \frac{1}{2N}\sum_{i=1}^{N}|(\hat{\mu}-\hat{\xi}_i)^{\top}x-\epsilon|+\frac{1}{2N}\sum_{i=1}^{N}((\hat{\mu}-\hat{\xi}_i)^{\top}x-\epsilon)+\epsilon\\
			\ & \quad \quad  {\rm{s.t.}} \ \hat{\mu}^{\top}x\geq \rho+\epsilon,
		\end{aligned}\right.
	\end{aligned}
	$$  
	Since $	\mathcal{F}_{\epsilon,\rho} =\{x | e^{\top}x=1,x\geq0,\mathbb{E}_{\widehat{\mathbb{P}}_N}(\xi^{\top}x)-\epsilon \|x\|_q\geq \rho\}$ and $\|\cdot\|_q=\|\cdot\|_1$, we have 
	$
	\mathcal{F}_{\epsilon,\rho} =\{x | e^{\top}x=1,x\geq0,\hat{\mu}^{\top}x-\epsilon \geq \rho\}.
	$
	Therefore, problem \eqref{EDRMLSAD} is equivalent to
	$$
	\begin{aligned}
		&\min_{x\in	\mathcal{X}}\ \frac{1}{2N}\sum_{i=1}^{N}|(\hat{\mu}-\hat{\xi}_i)^{\top}x-\epsilon|+\frac{1}{2N}\sum_{i=1}^{N}((\hat{\mu}-\hat{\xi}_i)^{\top}x-\epsilon)+\epsilon\\
		\ & \quad \quad  {\rm{s.t.}} \ \hat{\mu}^{\top}x\geq \rho+\epsilon,
	\end{aligned}
	$$    
	This concludes the proof.
	
	\section{Proof of Theorem \ref{thm3.1}}\label{App2}
	Obviously, the optimal solution $x^*$ satisfies the optimality condition $\mathbb{E}_{\mathbb{P}^*}\left[ h(\xi;x^*)\right] =0.$ 
	Since $h(\xi;x^*)  = g(x^*,\xi)-\lambda_1^*\xi-\lambda_2^*e-\lambda_3^*$, we have 
	$$
	\|h(\xi;x^*)\|_2^2\leq \| |g(x^*,\xi)|+|\lambda_1^*\xi|+|\lambda_2^*e|+|\lambda_3^*|\|_2^2,
	$$
	where 
	$$
	g(x^*,\xi):=\frac{-\xi}{1+\exp((\xi^{\top}x^*-\bar{\rho})/t)}. 
	$$
	Thus, $\mathbb{E}_{\mathbb{P}^*} \|h(\xi;x^*)\|_2^2<\infty$, which means $\mathbb{E}_{\mathbb{P}^*} \|h(\xi;x^*)\|_2^2$ is finite.
	while 
	$$
	\begin{aligned}
		D_{\xi}h(\xi,x^*) &=\frac{ x^*\xi^{\top}}{t\left[ 1+\exp((\xi^{\top}x^*-\bar{\rho})/t)\right] \left[ 1+\exp((-\xi^{\top}x^*+\bar{\rho})/t)\right] }\\
		&+ \frac{-I_m}{1+\exp((\xi^{\top}x^*-\bar{\rho})/t)}-\lambda_1^*I_m
	\end{aligned}
	$$
	is continuously differentiable, it also holds that
	$$
	\begin{aligned}
		&\mathbb{P}^*\left( \|\zeta^{\top}	D_{\xi}h(\xi,x^*) \|_1=0\right) \\
		&=\mathbb{P}^*\left( \begin{aligned}
			&t\left[1+\exp((-\xi^{\top}x^*+\bar{\rho})/t)\right] \zeta=(\zeta^{\top}\xi)x^*\\
			&-\lambda_1^*t\left[ 1+\exp((\xi^{\top}x^*-\bar{\rho})/t)\right] \left[ 1+\exp((-\xi^{\top}x^*+\bar{\rho})/t)\right] \zeta
		\end{aligned}\right)=0,
	\end{aligned}
	$$
	which implies $\mathbb{P}^*\left( \|\zeta^{\top}	D_{\xi}h(\xi,x^*) \|_1>0\right)>0.$
	
	Based on the above analysis, all the required assumptions of Theorem 4 in \cite{BJ.KY.MK.2016} are valid, so we can obtain the result in \cite[Theorem 4] {BJ.KY.MK.2016}:
	$$
	\sqrt{N}R_N(x^*)\stackrel{D}{\longrightarrow}\sup_{\zeta\in\bar{\Xi}}\zeta^{\top}H,
	$$
	where 
	$$ 
	H\sim\mathcal{N}\left( {\bf{0}}, Cov\left[h(\xi;x^*) \right] \right), \
	\bar{\mathbb{E}}=\left\lbrace\zeta\in\mathbb{R}^N:\|\zeta^{\top} D_{\xi}h(\xi,x^*)\|_1 \leq 1\right\rbrace.
	$$
	
	In the rest of the proof, we give the upper bound for $\bar{R}(1)$. First, we prove the fact that $\bar{\mathbb{E}}$ is a subset of the norm ball $\left\lbrace \zeta\in\mathbb{R}^N:\|\zeta\|_1\leq1\right\rbrace.$ By the H$\ddot{o}$lder's inequality $|\zeta^{\top}\xi|\leq\|\zeta\|_1\|\xi\|_\infty$, it is clearly observed that
	\begin{equation}\label{Hiden}
		\begin{aligned}
			& \|\zeta^{\top} D_{\xi}h(\xi,x^*)\|_1\\
			&=\left\| \frac{-\zeta^{\top}}{1+\exp((\xi^{\top}x^*-\bar{\rho})/t)}+\frac{(\zeta^{\top}\xi) (x^*)^{\top}}{t\left[ 1+\exp((\xi^{\top}x^*-\bar{\rho})/t)\right] \left[ 1+\exp((-\xi^{\top}x^*+\bar{\rho})/t)\right] }-\lambda_1^*\zeta^{\top}\right\| _1\\
			&\geq \left\| \frac{\zeta}{1+\exp((\xi^{\top}x^*-\bar{\rho})/t)}\right\| _1-\left\| \frac{(\zeta^{\top}\xi) x^*}{t\left[ 1+\exp((\xi^{\top}x^*-\bar{\rho})/t)\right] \left[ 1+\exp((-\xi^{\top}x^*+\bar{\rho})/t)\right] }\right\|_1-|\lambda_1^*|\|\zeta\|_1\\
			& \geq \left( \frac{1}{1+\exp((\xi^{\top}x^*-\bar{\rho})/t)}-\frac{\|\xi\|_\infty\|x^*\|_1}{t\left[ 1+\exp((\xi^{\top}x^*-\bar{\rho})/t)\right] \left[ 1+\exp((-\xi^{\top}x^*+\bar{\rho})/t)\right] }-|\lambda_1^*|\right) \|\zeta\|_1.
		\end{aligned}
	\end{equation}
	By \eqref{Hiden}, we know that if $\zeta\in\mathbb{R}^N$ is such that $\|\zeta\|_1=(1-\epsilon)^{-2}>1$ for a given $\epsilon>0,$ then $\|\zeta^{\top} D_{\xi}h(\xi,x^*)\|_1>1$ wherever
	$$
	\xi\in\widetilde{\Omega}_{\epsilon}:=\left\lbrace \xi:
	\begin{aligned} 
		&\frac{1}{1+\exp((\xi^{\top}x^*-\bar{\rho})/t)}\leq\frac{\epsilon}{2},\\
		& \frac{\|\xi\|_\infty\|x^*\|_1}{t\left[ 1+\exp((\xi^{\top}x^*-\bar{\rho})/t)\right] \left[ 1+\exp((-\xi^{\top}x^*+\bar{\rho})/t)\right] }+|\lambda_1^*| \geq 1-\frac{\epsilon}{2}
	\end{aligned}\right\rbrace 
	$$
	The set $\widetilde{\Omega}_{\epsilon}$ has positive probability for every $\epsilon>0$ because $\zeta$ has positive probability almost everywhere. Thus, we have $\bar{\mathbb{E}}\subset\left\lbrace \zeta\in\mathbb{R}^N:\|\zeta\|_1\leq1\right\rbrace.$
	Therefore,
	$$
	\bar{R}(1)=\max_{\zeta\in\bar{\mathbb{E}}}\left\lbrace\zeta^{\top}H  \right\rbrace\overset{D}{\leq} \max_{\|\zeta\|_1\leq 1}\left\lbrace\zeta^{\top}H  \right\rbrace:=\|\tilde{H}\|_\infty.
	$$
	
	In the remainder, we estimate $\tilde{H}$. Since our purpose is to find the appropriate radius $\epsilon$ in problem \eqref{DRMLSAD}, we need to analyze the case as $t\to0^{+}$.  When $t\to0^{+}$, we obtain
	$$
	\hat{g}(x^*,\xi):=\lim\limits_{t\to0^{+}}g(x^*,\xi)=
	\begin{cases}
		{\bf{0}}, \ &{\rm{if}}\ \xi^{\top}x^*-\bar{\rho}>0,\\
		-\frac{\xi}{2} , \ &{\rm{if}} \ \xi^{\top}x^*-\bar{\rho}=0,\\
		-\xi , \ &{\rm{if}} \ \xi^{\top}x^*-\bar{\rho}<0.
	\end{cases}
	$$
	Then we have 
	$$
	\mathbb{E}_{\mathbb{P}^*}\left[ \hat{g}^i(x^*,\xi)\right] = \mathbb{E}_{\mathbb{P}^*}\left( -\frac{\xi^i}{2}{\rm{1}}_{\left\lbrace \xi^{\top}x^*-\bar{\rho}=0 \right\rbrace} -\xi^i{\rm{1}}_{\left\lbrace \xi^{\top}x^*-\bar{\rho}<0\right\rbrace} \right) =- \mathbb{E}_{\mathbb{P}^*}\left( \xi^i{\rm{1}}_{\left\lbrace \xi^{\top}x^*-\bar{\rho}<0\right\rbrace} \right).
	$$
	Invoking \eqref{lambda1},\eqref{lambda2} and\eqref{lambda3}, we obtain
	$$
	\begin{aligned}
		\hat{\lambda}_1^*&:=\lim\limits_{t\to0^{+}}\lambda^*_1=\lim\limits_{t\to0^{+}} \frac{\mathbb{E}_{\mathbb{P}^*}\left[ g^i(x^*,\xi)\right]- (x^*)^{\top} \mathbb{E}_{\mathbb{P}^*}\left[  g(x^*,\xi)\right]}{\mathbb{E}_{\mathbb{P}^*}(\xi^i)-\bar{\rho}}\\
		&=\frac{-\mathbb{E}_{\mathbb{P}^*}\left( \xi^i{\bf{1}}_{\left\lbrace \xi^{\top}x^*-\bar{\rho}<0\right\rbrace }\right) +(x^*)^{\top}\mathbb{E}_{\mathbb{P}^*}\left( \xi{\bf{1}}_{\left\lbrace \xi^{\top}x^*-\bar{\rho}<0\right\rbrace }\right)}{\mathbb{E}_{\mathbb{P}^*}(\xi^i)-\bar{\rho}},\ i\in \mathcal{I}_*, 
	\end{aligned}
	$$
	and 
	$$
	\begin{aligned}
		&\hat{\lambda}_2^*:=\lim\limits_{t\to0^{+}}\lambda^*_2=\lim\limits_{t\to0^{+}}\left[ (x^*)^{\top} \mathbb{E}_{\mathbb{P}^*}\left[  g(x^*,\xi)\right]  -\lambda_1^*\bar{\rho}\right] =(x^*)^{\top}\mathbb{E}_{\mathbb{P}^*}\left( \xi{\bf{1}}_{\left\lbrace \xi^{\top}x^*-\bar{\rho}<0\right\rbrace }\right)  -\hat{\lambda}_1^*\bar{\rho},\\
		&\hat{\lambda}_3^*:=\lim\limits_{t\to0^{+}}\lambda_3^*=\lim\limits_{t\to0^{+}}\left[ \mathbb{E}_{\mathbb{P}^*}\left[ g(x^*,\xi)\right] -\lambda_1^*\mathbb{E}_{\mathbb{P}^*}(\xi)-\lambda_2^*e\right]\\
		&\quad=\mathbb{E}_{\mathbb{P}^*}\left( \xi{\bf{1}}_{\left\lbrace \xi^{\top}x^*-\bar{\rho}<0\right\rbrace }\right) -\hat{\lambda}_1^*\mathbb{E}_{\mathbb{P}^*}(\xi)-\hat{\lambda}_2^*e.
	\end{aligned}
	$$
	Thus we can get 
	$$
	\hat{h}(\xi;x^*)=\lim\limits_{t\to0^{+}}h(\xi;x^*)=\hat{g}(x^*,\xi)-\hat{\lambda}_1^*\xi-\hat{\lambda}_2^*e-\hat{\lambda}_3^*,
	$$
	and 
	$$
	\begin{aligned}
		\|\hat{h}(\xi;x^*)\|_2^2&\leq\| |\hat{g}(x^*,\xi)|+|\hat{\lambda}_1^*\xi|+|\hat{\lambda}_2^*e|+|\hat{\lambda}_3^*|\|_2^2\\
		&\leq \| |\xi|+|\hat{\lambda}_1^*||\xi|+|\hat{\lambda}_2^*|e+|\hat{\lambda}_3^*|\|_2^2\\
		&=\|\left( 1+|\hat{\lambda}_1^*|\right) |\xi|+|\hat{\lambda}_2^*|e+\hat{\lambda}_3^*\|_2^2,
	\end{aligned}
	$$
	which implies $\mathbb{E}_{\mathbb{P}^*}\|\hat{h}(\xi;x^*)\|_2^2\leq\mathbb{E}_{\mathbb{P}^*}\left[ \|\left( 1+|\hat{\lambda}_1^*|\right) |\xi|+|\hat{\lambda}_2^*|e+\hat{\lambda}_3^*\|_2^2\right] $.
	Let 
	$$
	\tilde{H}\sim\mathcal{N}\left( {\bf{0}}, Cov_{\mathbb{P}^*}\left[(1+|\hat{\lambda}_1^*|)|\xi|+|\hat{\lambda}_2^*|e+\hat{\lambda}_3^* \right] \right),
	$$
	then $Cov\left[ \tilde{H}\right] -Cov\left[ H\right]
	$
	is positive definite. Therefore, $\bar{R}(1)$ is stochastically dominated by $\|\tilde{H}\|_\infty$.
	
	
	
	
\end{appendices}


\bibliography{sn-bibliography}

\end{document}